\documentclass[a4paper]{amsart}
\usepackage{amsthm,amsfonts,amsmath,amssymb}
\usepackage[abs]{overpic}
\usepackage{comment} 

\newtheorem{theorem}{Theorem}[section]
\newtheorem{lemma}[theorem]{Lemma}

\newtheorem{proposition}[theorem]{Proposition}

\theoremstyle{definition}
\newtheorem{definition}[theorem]{Definition}

\newtheorem{remark}[theorem]{Remark}

\theoremstyle{remark}

\newcommand{\e}{\varepsilon}

\makeatletter

\@addtoreset{figure}{section}
\makeatother

\makeatletter
  
  \@addtoreset{equation}{section}
\makeatother

\makeatletter
\@namedef{subjclassname@2020}{%
  \textup{2020} Mathematics Subject Classification}
\makeatother

\begin{document}
\title[On invariants of multiplexed virtual links]{On invariants of multiplexed virtual links}

\author[Kodai Wada]{Kodai Wada}
\address{Department of Mathematics\\ Kobe University\\ 1-1 Rokkodai, Nada-ku\\ Kobe 657-8501, Japan}
\email{wada@math.kobe-u.ac.jp}



\thanks{This work was supported by JSPS KAKENHI Grant Numbers JP21K20327 and JP23K12973.}



\begin{abstract}
For a virtual knot $K$ and an integer $r$ with $r\geq2$, we introduce a method of constructing an $r$-component virtual link $L(K;r)$, which we call the $r$-multiplexing of $K$. 
Every invariant of $L(K;r)$ is an invariant of $K$. 
We give a way of calculating three kinds of invariants of $L(K;r)$ using invariants of $K$. 
As an application of our method, we also show that Manturov's virtual $n$-colorings for $K$ can be interpreted 
as certain classical $n$-colorings for $L(K;2)$. 
\end{abstract}

\maketitle

\section{Introduction}
Virtual links were defined by Kauffman~\cite{Kau} as a generalization of classical links in the $3$-sphere. 
Let $r$ be a positive integer. 
An \emph{$r$-component virtual link diagram} is the image of an immersion of $r$ ordered and oriented circles into the plane, whose singularities are only transverse double points. 
Such double points are divided into \emph{positive}, \emph{negative}, and \emph{virtual crossings} as shown in Figure~\ref{xing}. 
A positive/negative crossing is also called a \emph{real crossing}. 

\begin{figure}[htbp]
\centering
  \begin{overpic}[width=6cm]{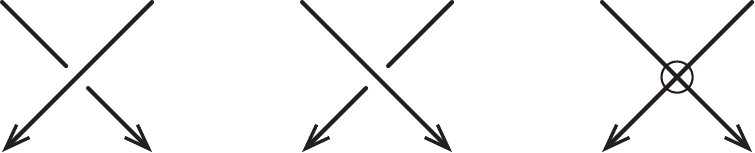}
    \put(1,-15){positive}
    \put(67.1,-15){negative}
    \put(139,-15){virtual}
  \end{overpic}
\vspace{1em}
\caption{Types of crossings}
\label{xing}
\end{figure}

Two virtual link diagrams are said to be \emph{equivalent} if they are related by a finite sequence of \emph{generalized Reidemeister moves}, which consist of classical Reidemeister moves R1--R3 and virtual Reidemeister moves VR1--VR4 as shown in Figure~\ref{gReid}. An \emph{$r$-component virtual link} is an equivalence class of $r$-component virtual link diagrams. 
In particular, a $1$-component virtual link (diagram) is called a \emph{virtual knot} (\emph{diagram}). 

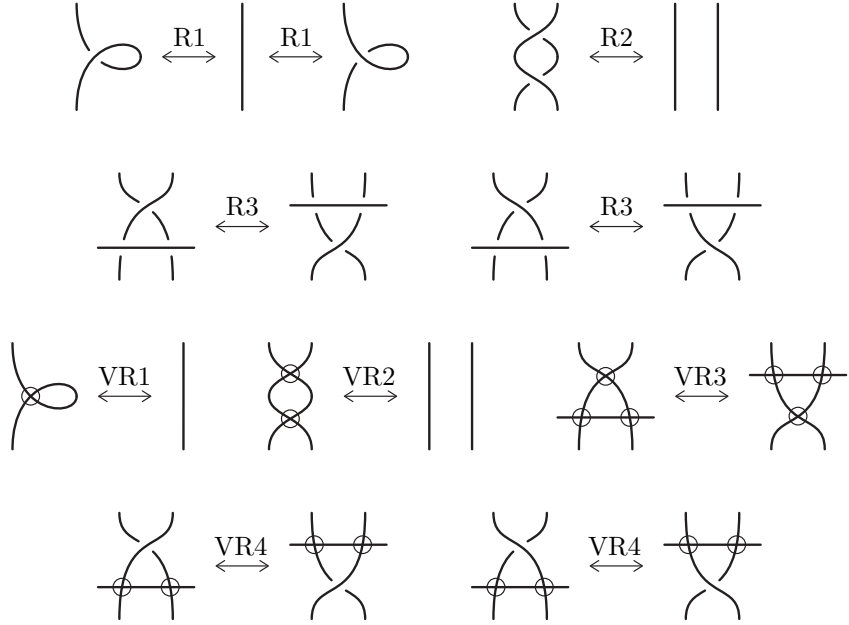
\begin{figure}[htbp]
\centering
  \begin{overpic}[width=11cm]{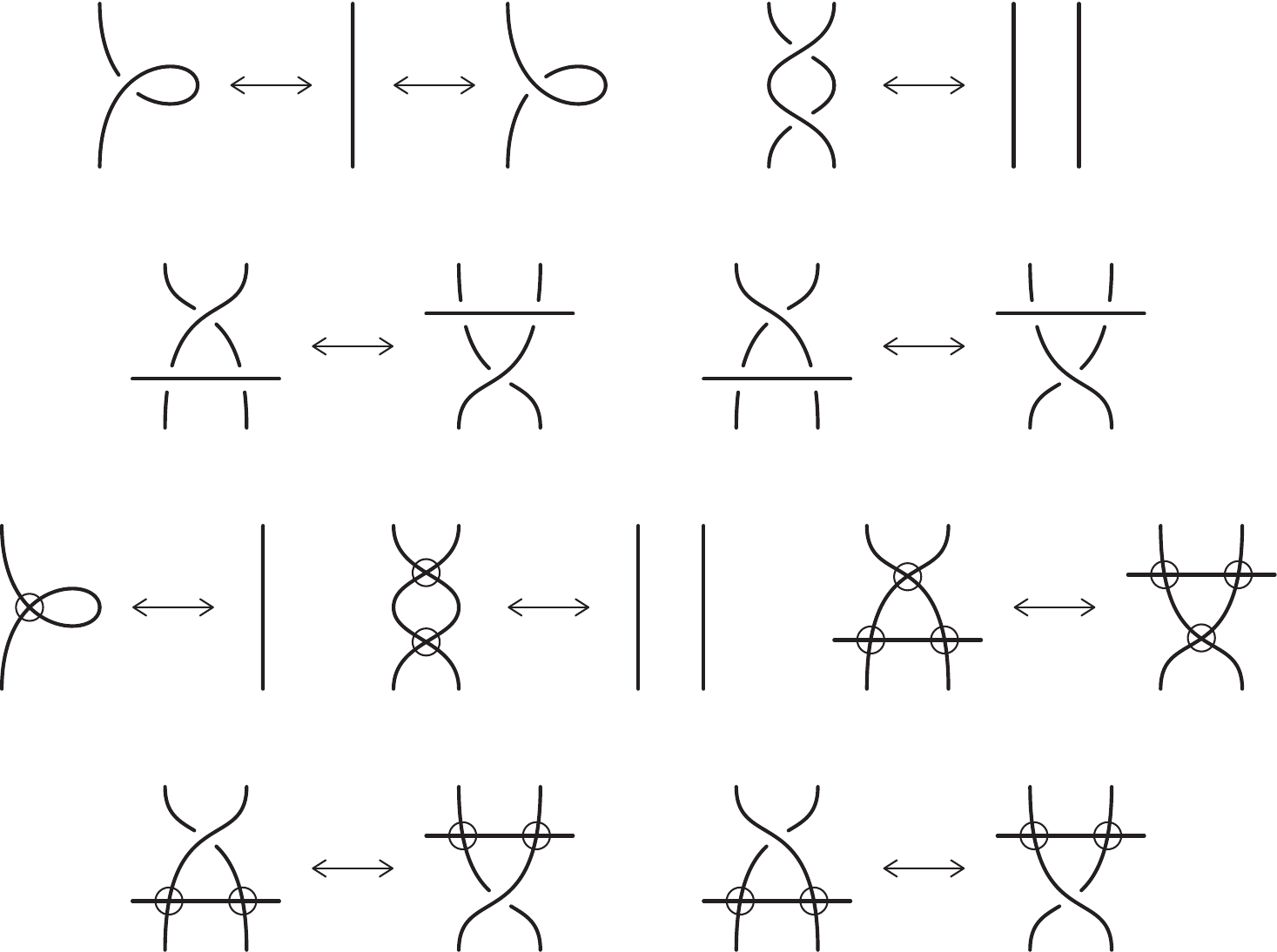}
    \put(60.6,216.7){R1}
    \put(100.7,216.7){R1}
    \put(220.4,216.7){R2}
    \put(80.3,152.7){R3}
    \put(220.4,152.7){R3}
    \put(32.5,88.6){VR1}
    \put(124,88.6){VR2} 
    \put(248,88.6){VR3}
    \put(76,24.6){VR4}
    \put(216,24.6){VR4}  
  \end{overpic}
\caption{Generalized Reidemeister moves}
\label{gReid}
\end{figure}

This paper has two purposes. 
The first is to introduce a method of constructing an $r$-component virtual link diagram $L(D;r)$ 
from a virtual knot diagram $D$ for an integer $r$ with $r\geq2$. 
We call this method the \emph{$r$-multiplexing} of $D$. 
See Section~\ref{sec-multiplexing} for the precise definition. 
In this section we will show the following. 

\begin{theorem}\label{thm-multiplexing}
If two virtual knot diagrams $D$ and $D'$ are equivalent, then $L(D;r)$ and $L(D';r)$ are equivalent for any integer $r$ with $r\geq2$. 
\end{theorem}

Let $D$ be a virtual knot diagram equipped with an additional information called a \textit{cut system~$P$}. 
An \emph{$r$-fold cyclic covering diagram $\varphi_{r}(D,P)$} of the pair $(D,P)$ is 
an $r$-component virtual link diagram introduced by Kamada~\cite{Kam19}, 
who proved that if two virtual knot diagrams $D$ and $D'$ are equivalent, then so are $\varphi_{r}(D, P)$ and $\varphi_{r}(D', P')$ for any cut systems $P$ and $P'$ of $D$ and $D'$, respectively. 
We can see that $L(D;r)$ is equivalent to $\varphi_{r}(D,P)$ for any cut system $P$. 
Thus although the diagrams $L(D;r)$ and $\varphi_{r}(D,P)$ yield the same virtual link, 
our method is more simple than Kamada's one 
in the sense that it does not use cut systems. 
The author believes that it is worth proposing the $r$-multiplexing construction. 

Let $K$ be a virtual knot. 
We denote by $L(K;r)$ the $r$-component virtual link represented by $L(D;r)$ constructed from a diagram $D$ of $K$. 
It is called the \emph{$r$-multiplexed virtual link} of $K$. 
By Theorem~\ref{thm-multiplexing}, every invariant of $L(K;r)$ is an invariant of $K$. 
Therefore it is expected that new invariants of $K$ are obtained from known invariants of $L(K;r)$. 
In fact, it was demonstrated by Kamada~\cite{Kam19-talk} that the Kishino knot $K$ and the trivial knot~$O$ can be distinguished by the Jones polynomials of $L(K;3)$ and $L(O;3)$ 
via the $3$-fold cyclic covering $\varphi_{3}$, 
although $K$ and $O$ have the same Jones polynomial. 
It would be interesting to develop such a study further. 

In a different direction, Theorem~\ref{thm-multiplexing} motivates us to find a way of understanding and calculating 
invariants of $L(K;r)$ by means of invariants of $K$.  
In Section~\ref{sec-inv} we study three kinds of invariants of $L(K;r)=K_{1}\cup\cdots\cup K_{r}$: 
the $(i,j)$-linking number $\mathrm{Lk}(K_{i},K_{j})\in\mathbb{Z}$~\cite{GPV}, 
the $i$th $n$-writhe $J_{n}^{i}(L(K;r))\in\mathbb{Z}$~\cite{Xu} and the knot type of $K_{i}$. 
The second purpose of this paper is to provide a way of describing these three invariants, in terms of invariants of $K$, 
by proving Theorems~\ref{thm-lk}, \ref{thm-self-writhe} and \ref{thm-knot-type}. 

The \emph{$n$-writhe $J_{n}(K)$} of $K$  
is an integer-valued invariant of $K$ for any nonzero integer~$n$ 
defined by Satoh and Taniguchi~\cite{ST}. 
We will show that the $(i,j)$-linking number $\mathrm{Lk}(K_{i},K_{j})$ and 
the $i$th $n$-writhe $J_{n}^{i}(L(K;r))$ of $L(K;r)=K_{1}\cup\cdots\cup K_{r}$ can be calculated from $J_{n}(K)$ as follows: 

\begin{theorem}\label{thm-lk}
Let $K$ be a virtual knot and $r$ an integer with $r\geq2$. 
Then the $r$-multiplexed virtual link $L(K;r)=K_{1}\cup\dots\cup K_{r}$ of $K$ satisfies 
\[
\mathrm{Lk}(K_{i},K_{j})=\displaystyle\sum_{n\equiv i-j\hspace{-0.5em}\pmod{r}}J_{n}(K)
\] 
for any integers $i$ and $j$ with $1\leq i\neq j\leq r$. 
\end{theorem} 

\begin{theorem}\label{thm-self-writhe}
Let $K$ be a virtual knot and $r$ an integer with $r\geq2$. 
Then the $r$-multiplexed virtual link $L(K;r)$ of $K$ satisfies 
\[
J_{n}^{i}(L(K;r))=
\begin{cases}
J_{n}(K) & \text{if } n\equiv0\pmod{r}, \\
0 & \text{otherwise} 
\end{cases}
\] 
for any integer $i$ with $1\leq i\leq r$. 
In particular, $J_{n}^{i}(L(K;r))=J_{n}^{j}(L(K;r))$ for any integers $i$ and $j$ with $1\leq i\neq j\leq r$. 
\end{theorem}

The \emph{$r$th covering $K^{(r)}$} of $K$ is a virtual knot constructed from $K$ for a positive integer $r$, 
which is an invariant of $K$~\cite{IK,NNS}. 
The $r$th covering was originally considered for flat virtual knots by Turaev~\cite{Tur}, 
where a \emph{flat virtual knot}~\cite{Kau} is an equivalence class of virtual knots modulo crossing changes. 
(In~\cite{Tur} flat virtual knots are called \emph{virtual strings}.)  
We will show that the knot type of every $i$th component $K_{i}$ of the $r$-multiplexed virtual link $L(K;r)$ can be determined by $K^{(r)}$ as follows: 

\begin{theorem}\label{thm-knot-type}
Let $K$ be a virtual knot and $r$ an integer with $r\geq2$. 
Then the knot type of the $i$th component $K_{i}$ of $L(K;r)$ coincides with 
that of the $r$th covering $K^{(r)}$ of $K$ 
for any integer $i$ with $1\leq i\leq r$. 
In particular, $K_{1}, \dots, K_{r}$ are the same virtual knot. 
\end{theorem}

A \emph{virtual quandle} is a quandle \cite{Joy,Mat} equipped with a unary operation introduced by Manturov~\cite{Man}.
Adding two new relations to every virtual crossing by the unary operation, he defined \emph{virtual quandle colorings} of a virtual knot diagram as a generalization of quandle colorings. 

In Section~\ref{sec-coloring} we define \emph{virtual $n$-colorings} of a virtual knot diagram $D$ as a special case of virtual quandle colorings. 
See Definition~\ref{def-vcoloring}.  
Denote by $\mathrm{Col}_{n}^{v}(D)$ the set of virtual $n$-colorings of $D$. 
We will give an interpretation of $\mathrm{Col}_{n}^{v}(D)$ in terms of a certain subset $\mathrm{Col}_{n}^{0}(L(D;2))$ of the set of (classical) $n$-colorings of $L(D;2)$ as follows: 

\begin{theorem}\label{thm-coloring}
Let $D$ be a virtual knot diagram. 
Then there is a bijection between $\mathrm{Col}_{n}^{v}(D)$ and $\mathrm{Col}_{n}^{0}(L(D;2))$. 
\end{theorem}

\section{Construction of multiplexed virtual links}\label{sec-multiplexing}
In this section we introduce the $r$-multiplexing of a virtual knot diagram, 
which is a method of constructing an $r$-component virtual link diagram, 
and prove Theorem~\ref{thm-multiplexing}. 

Let $D$ be a virtual knot diagram and $r$ an integer with $r\geq2$.  
Then a virtual link diagram, denoted by $L(D;r)$, 
is constructed by taking $r$ parallel copies of $D$ such that every crossing of $D$ is replaced with a set of crossings as shown in Figure~\ref{multiplexing}. 
More precisely, a positive (or negative) crossing of $D$ is replaced with $r$~positive (or negative) crossings and $r^{2}-r$ virtual crossings. 
A virtual crossing of $D$ is replaced with $r^{2}+2r-2$ virtual crossings. 
We remark that the number of components of $L(D;r)$ is equal to $r$. 
We call this method of constructing $L(D;r)$ the \emph{$r$-multiplexing} of $D$. 
Figure~\ref{ex-3multiplexing} shows an example with $r=3$.

\begin{figure}[htbp]
\centering
\vspace{1em}
  \begin{overpic}[width=8cm]{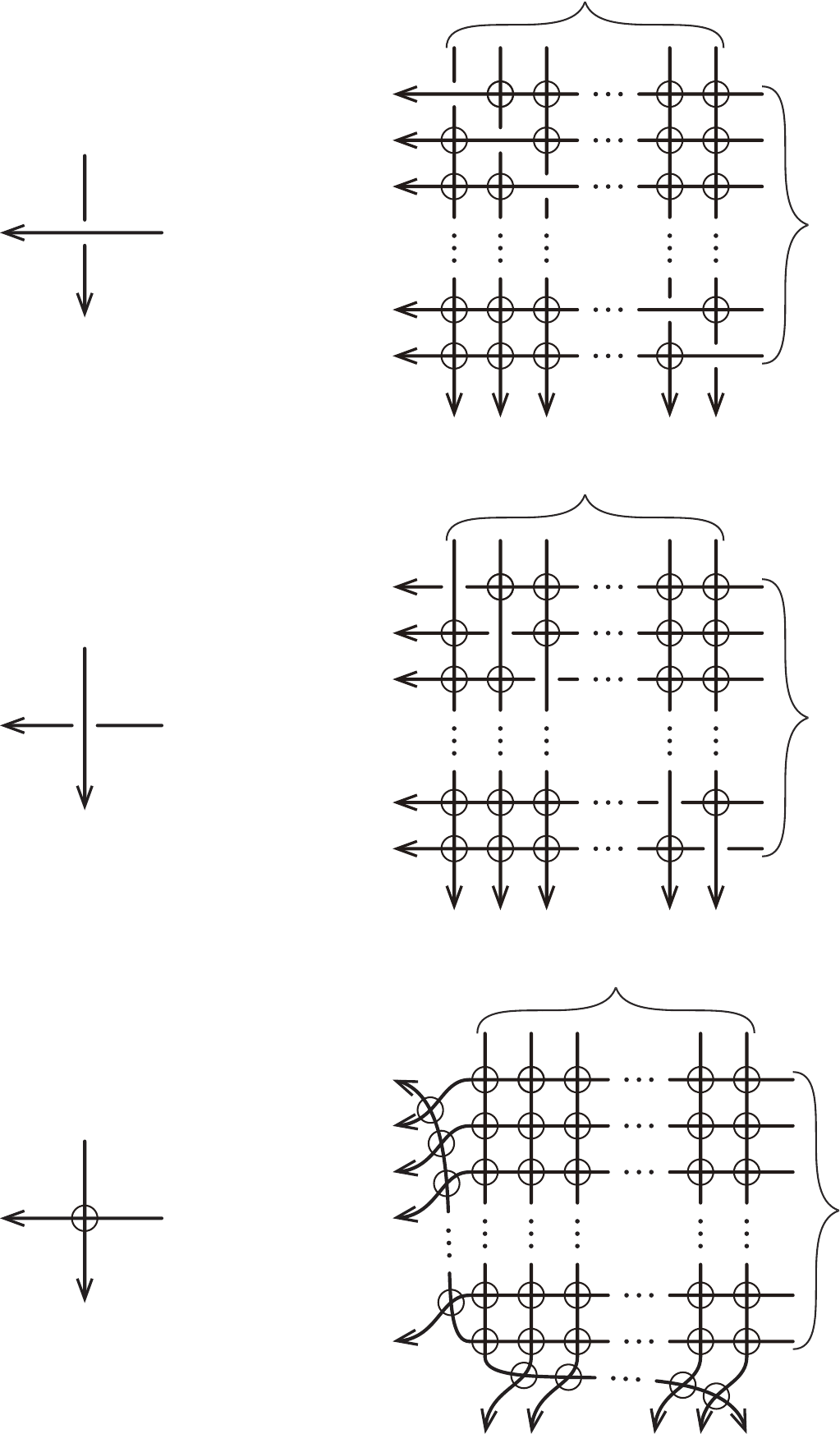}
    \put(-12.5,291){positive crossing}
    \put(-15,158){negative crossing}
    \put(-8.3,23.5){virtual crossing}
    \put(18.5,-14){$D$}
    \put(150,-14){$L(D;r)$}
    \put(156.5,390){$r$}
    \put(221,325){$r$}
    \put(156.5,256.5){$r$}
    \put(221,191.5){$r$}
    \put(164.5,123){$r$}
    \put(229,58){$r$} 
  \end{overpic}
\vspace{1em}
\caption{The $r$-multiplexing of a virtual knot diagram $D$}
\label{multiplexing}
\end{figure}

\begin{figure}[htbp]
\centering
  \begin{overpic}[width=9cm]{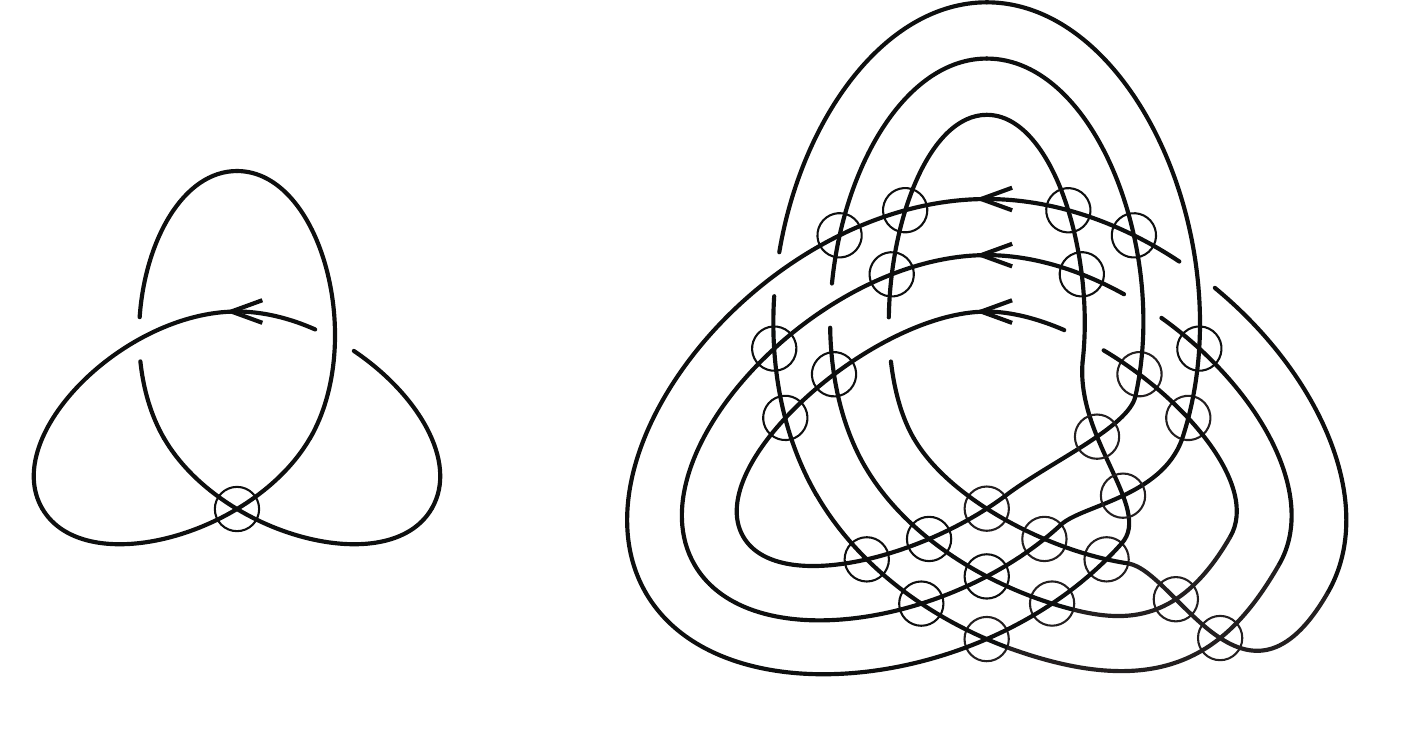}
    \put(38.5,18.5){$D$}
    \put(163,-5){$L(D;3)$}
  \end{overpic}
\vspace{1em}
\caption{A virtual knot diagram $D$ and its $3$-multiplexed virtual link diagram $L(D;3)$}
\label{ex-3multiplexing}
\end{figure}

\begin{remark}
A \emph{classical knot diagram} is a virtual knot diagram without virtual crossings. 
If $D$ is a classical knot diagram, 
then all the nonself-crossings of the $r$-multiplexed virtual link diagram $L(D;r)$ of $D$ are virtual crossings. 
Therefore $L(D;r)$ can be deformed into a disjoint union of $r$ copies of $D$ by a finite sequence of virtual Reidemeister moves VR1--VR4. 
In fact, the Gauss diagram of $L(D;r)$ coincides with a disjoint union of $r$ copies of the Gauss diagram of $D$. 
Refer to~\cite{GPV} for the definition of Gauss diagrams. 
\end{remark}

To give the proof of Theorem~\ref{thm-multiplexing}, 
we prepare six lemmas (Lemmas~\ref{lem-R1}--\ref{lem-VR4}). 
The first three lemmas deal with four moves R1a, R1b, R2a and R3a as shown in Figure~\ref{Polyak-set}, 
which form a minimal generating set of oriented classical Reidemeister moves R1--R3~\cite{Pol}. 
The fourth and fifth lemmas will be used to prove the last one concerning the virtual Reidemeister move VR4. 

\begin{figure}[htbp]
\centering
  \begin{overpic}[width=12cm]{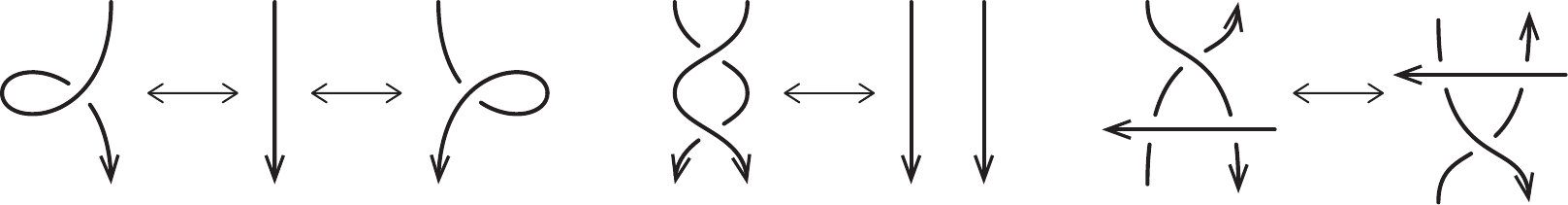}
    \put(33.2,28.3){R1a}
    \put(68.5,28.3){R1b} 
    \put(171.8,28.3){R2a}
    \put(282.7,28.3){R3a}
  \end{overpic}
\caption{A minimal generating set of classical Reidemeister moves}
\label{Polyak-set}
\end{figure}
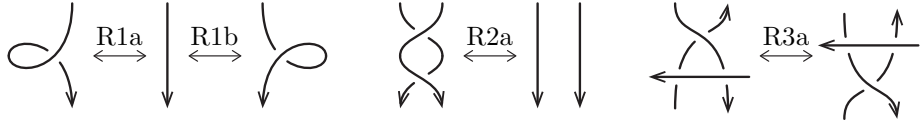

\begin{lemma}\label{lem-R1}
If two virtual knot diagrams $D$ and $D'$ are related by an R1a or R1b move, 
then $L(D;r)$ and $L(D';r)$ are equivalent for any integer $r$ with $r\geq2$. 
\end{lemma}

\begin{proof}
We only prove the result for the case $r=3$; 
the other cases are shown similarly. 

Assume that $D$ and $D'$ are related by an R1a move. 
Figure~\ref{pf-lem-R1} shows that $L(D;3)$ and $L(D';3)$ are related by a finite sequence of three R1a moves and several VR2 moves, 
where the symbol ``$\sim$'' means a finite sequence of virtual Reidemeister moves VR1--VR4. 

If $D$ and $D'$ are related by an R1b move, 
then $L(D;3)$ and $L(D';3)$ are related by three R1b moves and several VR2 moves. 
This follows from the sequence in the figure with the orientations of all the strings reversed. 
\end{proof}

\begin{figure}[htbp]
\centering
  \begin{overpic}[width=10cm]{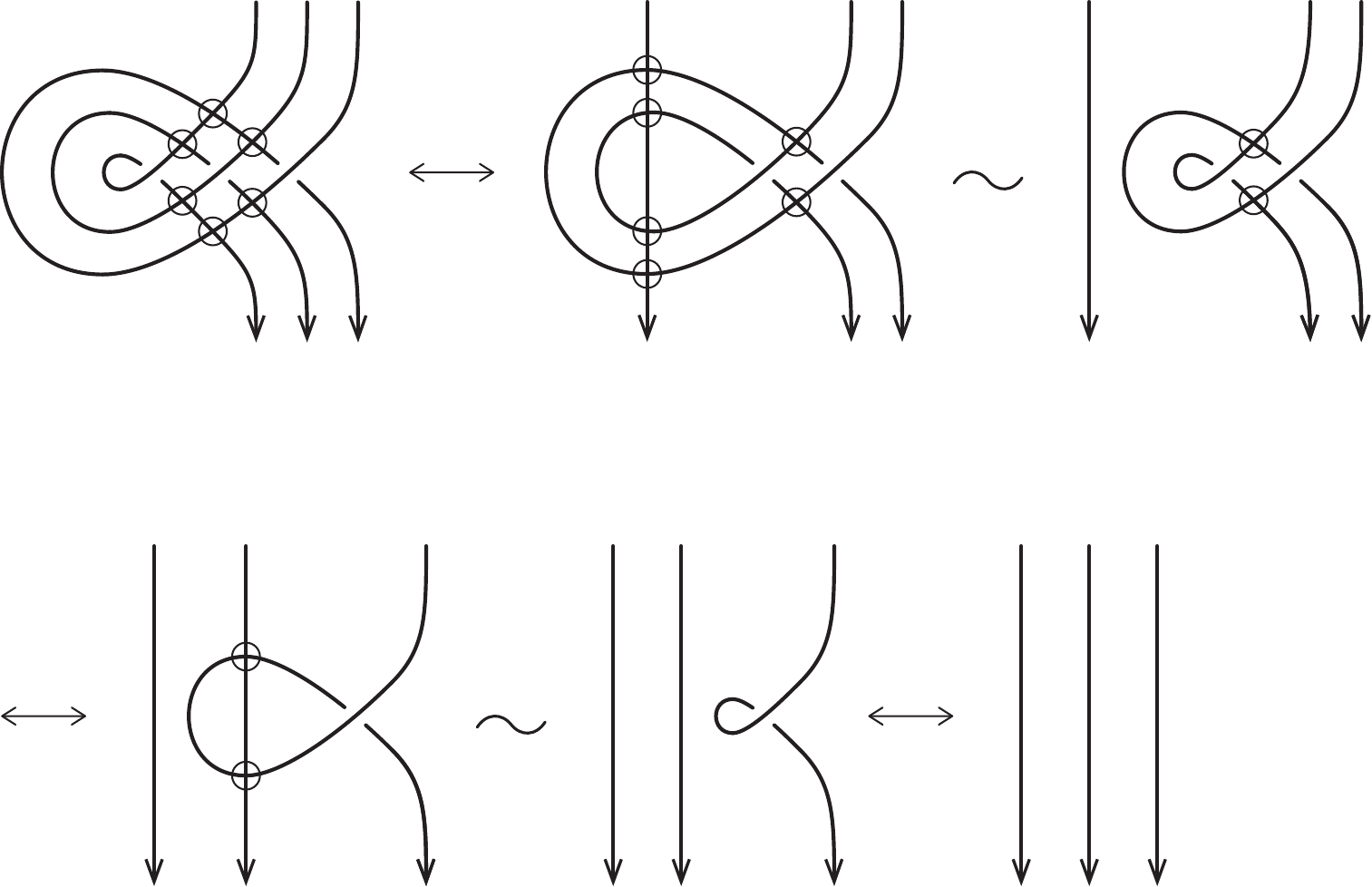}
    \put(85,152){R1a}
    \put(0,39.3){R1a}
    \put(180,39.3){R1a}
    \put(22,96){$L(D;3)$}
    \put(208.5,-16.7){$L(D';3)$}
  \end{overpic}
\vspace{1em}
\caption{Proof of Lemma~\ref{lem-R1} for $r=3$}
\label{pf-lem-R1}
\end{figure}

\begin{lemma}\label{lem-R2}
If two virtual knot diagrams $D$ and $D'$ are related by an R2a move, 
then $L(D;r)$ and $L(D';r)$ are equivalent for any integer $r$ with $r\geq2$. 
\end{lemma}

\begin{proof}
The result for the case $r=3$ follows from Figure~\ref{pf-lem-R2}, 
which shows that $L(D;3)$ and $L(D';3)$ are related by a finite sequence of three R2a moves and several VR2 moves. 
The other cases are shown similarly. 
\end{proof}

\begin{figure}[htbp]
\centering
  \begin{overpic}[width=12cm]{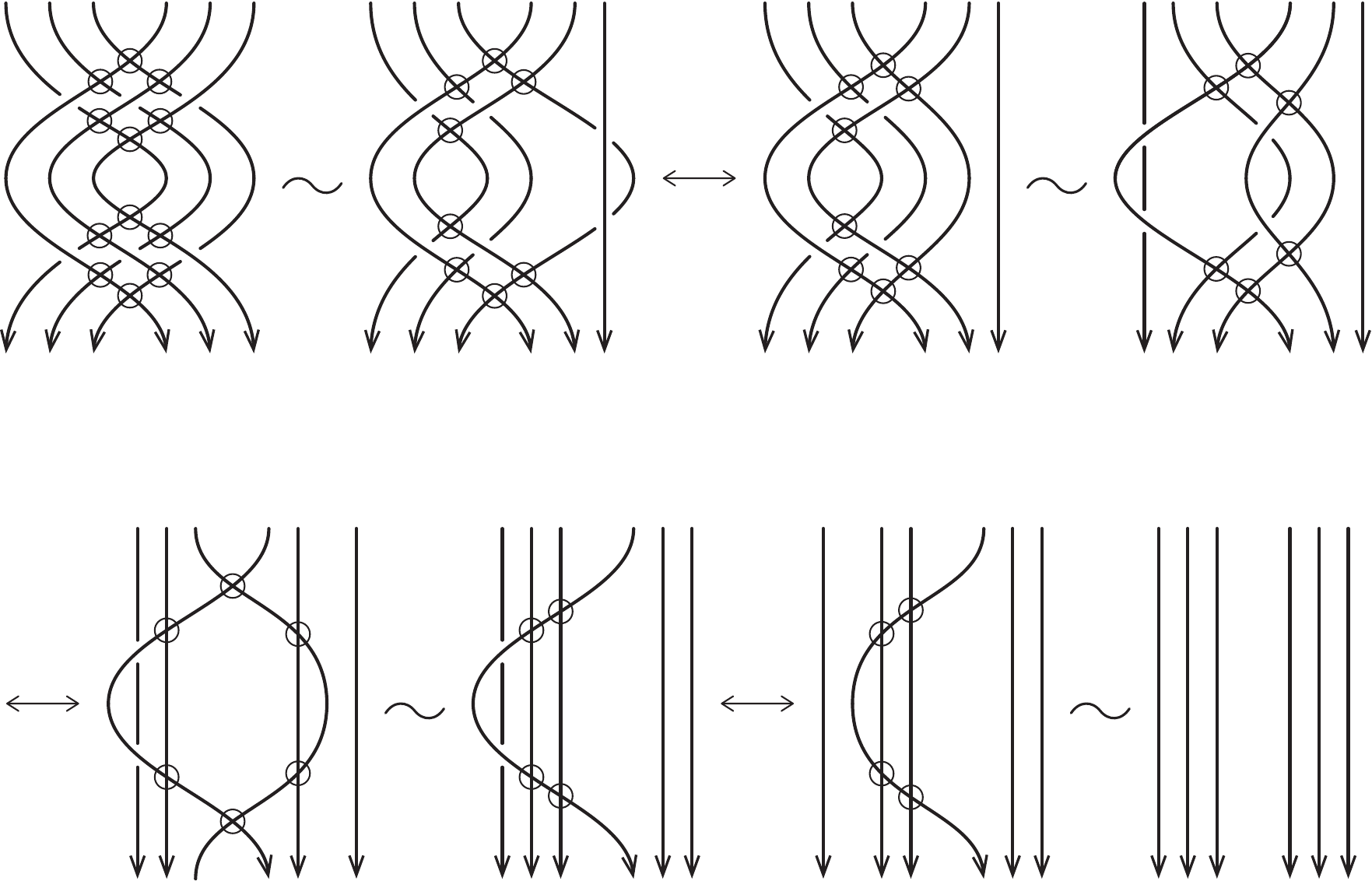}
    \put(165.3,178.3){R2a}
    \put(2,47.5){R2a}
    \put(179.8,47.5){R2a}
    \put(16,114){$L(D;3)$}
    \put(294,-17){$L(D';3)$}
  \end{overpic}
\vspace{1em}
\caption{Proof of Lemma~\ref{lem-R2} for $r=3$}
\label{pf-lem-R2}
\end{figure}

\begin{lemma}\label{lem-R3}
If two virtual knot diagrams $D$ and $D'$ are related by an R3a move, 
then $L(D;r)$ and $L(D';r)$ are equivalent for any integer $r$ with $r\geq2$. 
\end{lemma}

\begin{proof}
The result for the case $r=3$ follows from Figure~\ref{pf-lem-R3}, 
which shows that $L(D;3)$ and $L(D';3)$ are related by a finite sequence of three R3a moves and several VR3 and VR4 moves. 
The other cases are shown similarly. 
\end{proof}

\begin{figure}[htbp]
\centering
  \begin{overpic}[width=12cm]{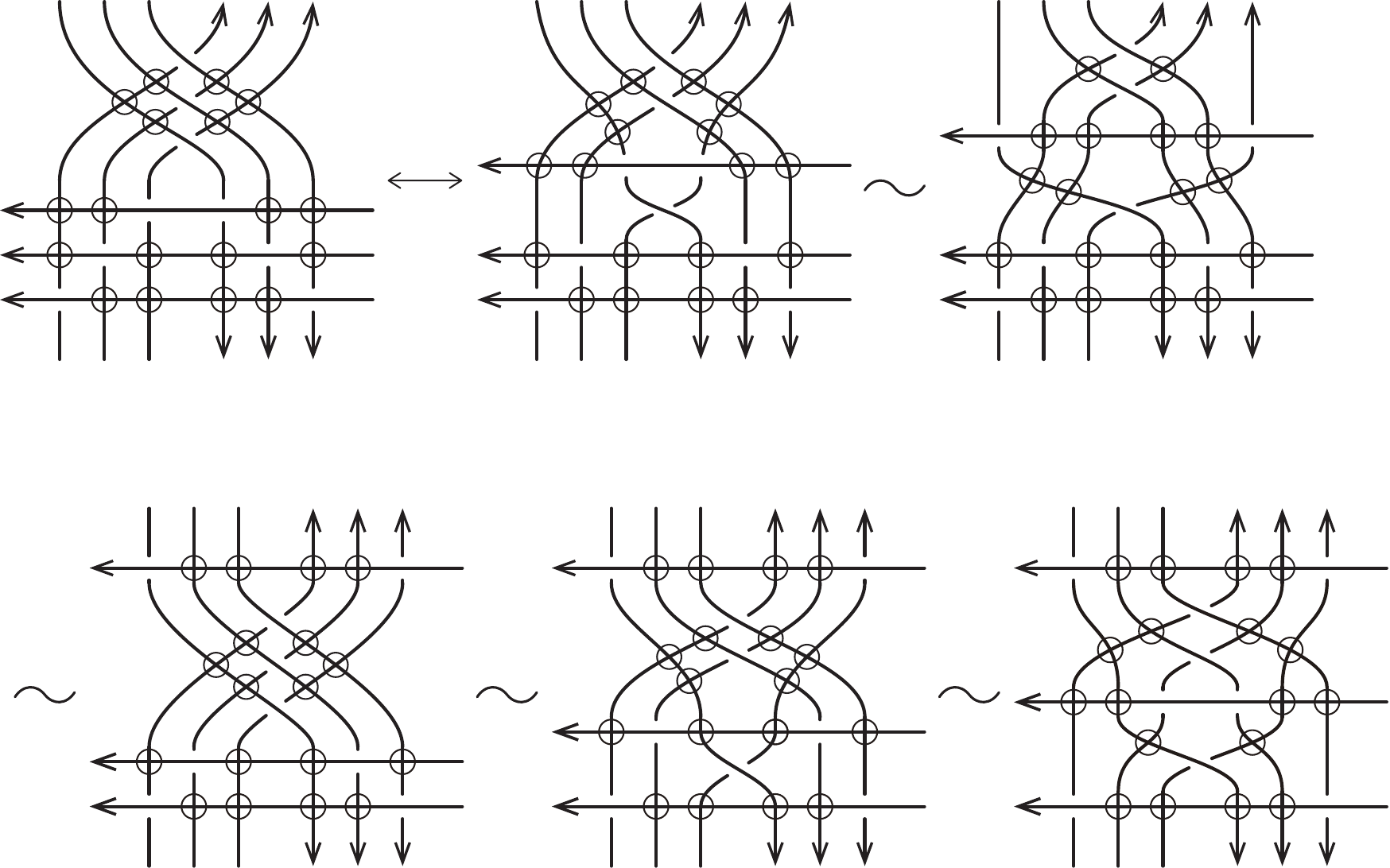}
    \put(95.5,172.5){R3a}
    \put(30,109.5){$L(D;3)$}
  \end{overpic}

\vspace{3.5em}

  \begin{overpic}[width=12cm]{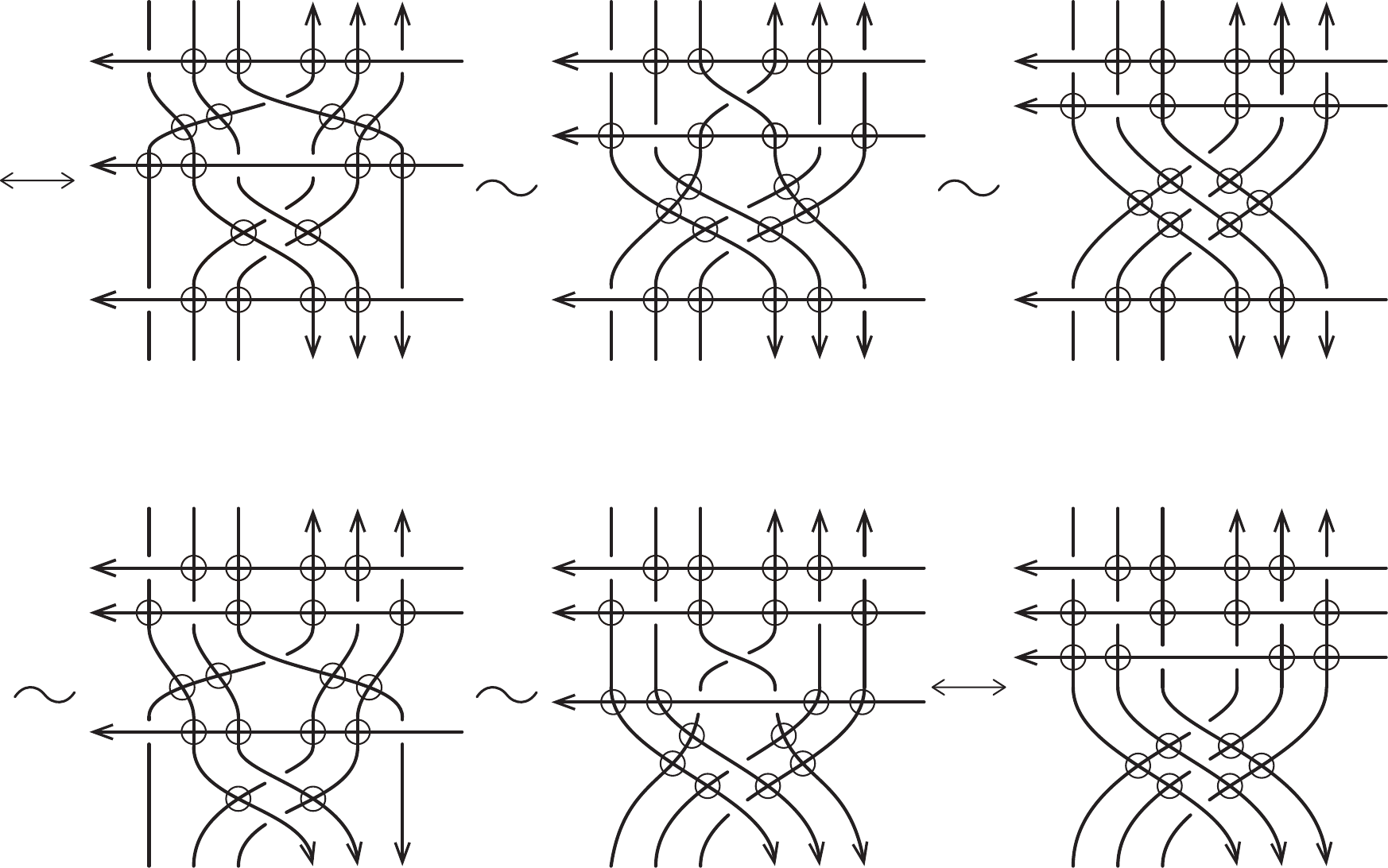}
    \put(0.5,173){R3a}
    \put(229.5,48.3){R3a}
    \put(278,-14.5){$L(D';3)$}
  \end{overpic}
\vspace{1em}
\caption{Proof of Lemma~\ref{lem-R3} for $r=3$}
\label{pf-lem-R3}
\end{figure}

\begin{lemma}\label{lem-M1toM4}
Let $r$ be an integer with $r\geq2$. 
Each of the four moves $M_{1}(r)$--$M_{4}(r)$ in Figure~\ref{M1toM4} is realized by a finite sequence of virtual Reidemeister moves. 
\end{lemma}

\begin{figure}[htbp]
\vspace{1em}
\centering
  \begin{overpic}[width=9cm]{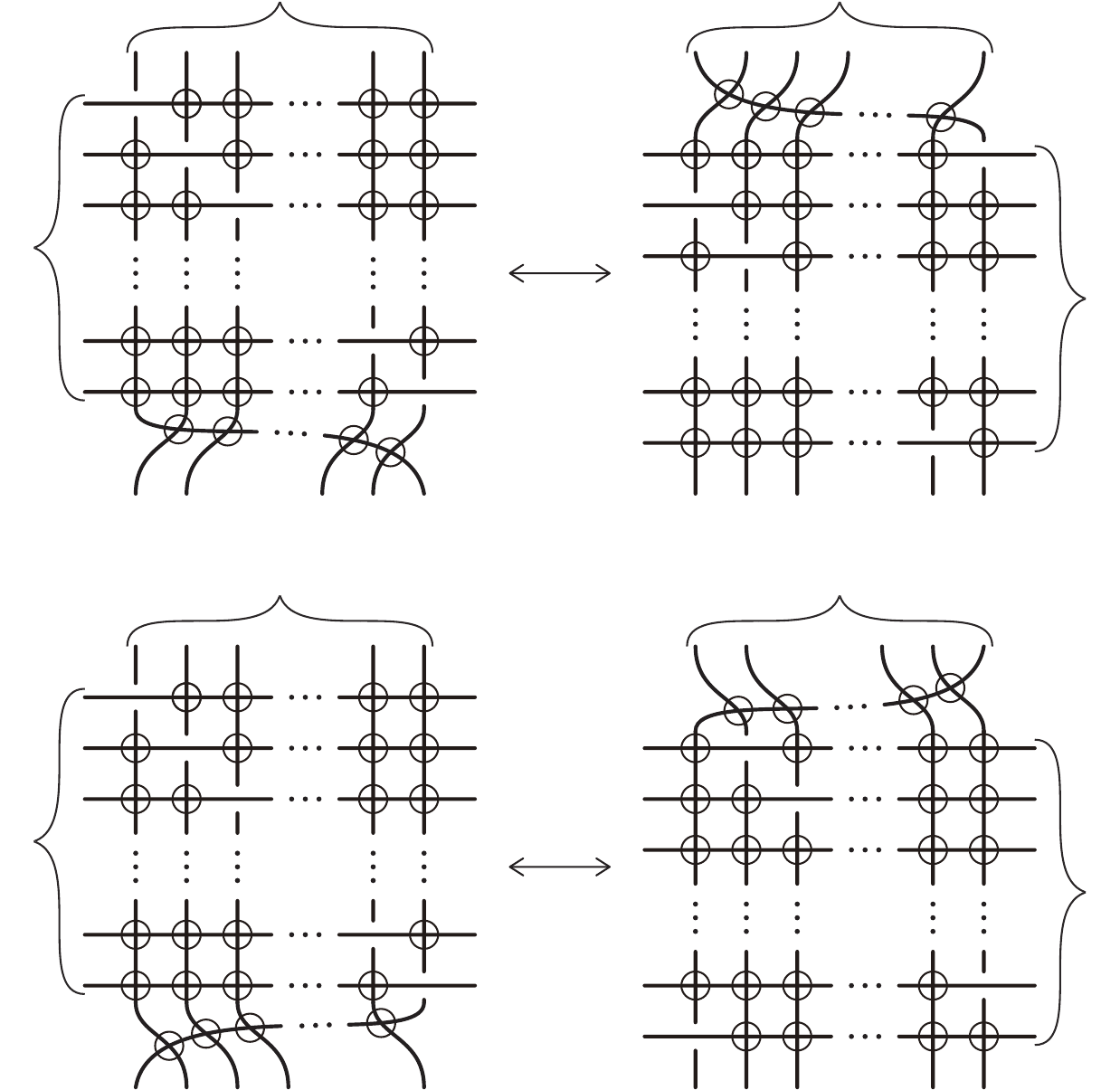}
    \put(114.4,192.5){$M_{1}(r)$}
    \put(114.4,57.5){$M_{2}(r)$}
    \put(62,250.8){$r$}
    \put(190,250.8){$r$}
    \put(1.4,190.8){$r$}
    \put(251,179){$r$}
    \put(62,115){$r$}
    \put(190,115){$r$}
    \put(1.4,55){$r$}
    \put(251,43.2){$r$}
  \end{overpic}

\vspace{2em}

  \begin{overpic}[width=9cm]{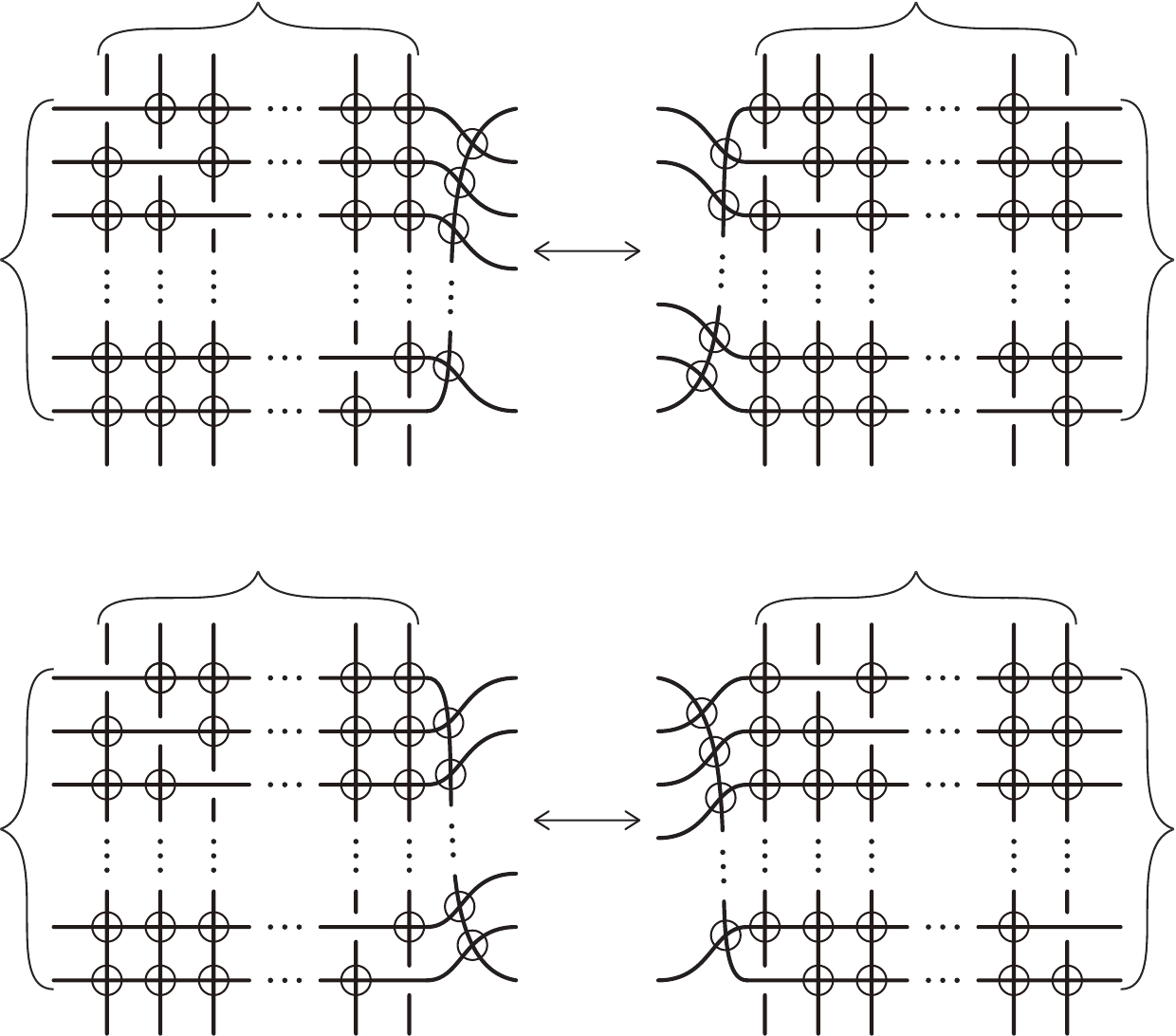}
    \put(115,177){$M_{3}(r)$}
    \put(114.4,53){$M_{4}(r)$}
    \put(54.2,227.3){$r$}
    \put(197.5,227.3){$r$}
    \put(-6.3,167.3){$r$}
    \put(257.5,167.3){$r$}
    \put(54.2,103.3){$r$}
    \put(197.5,103.3){$r$}
    \put(-6.3,43.3){$r$}
    \put(257.5,43.3){$r$}
   \end{overpic}
\caption{Moves $M_{1}(r)$--$M_{4}(r)$, each of which involves $r$ real crossings and $r^{2}-1$ virtual crossings}
\label{M1toM4}
\end{figure}

\begin{proof}
The result for the case $r=3$ follows from Figure~\ref{pf-lem-M1toM4}, 
which shows that for each $i\in\{1,2,3,4\}$, 
an $M_{i}(3)$ move is realized by a finite sequence of several VR3 and VR4 moves. 
The other cases are shown similarly. 
\end{proof}

\begin{figure}[htbp]
\vspace{2em}
\centering
  \begin{overpic}[width=11cm]{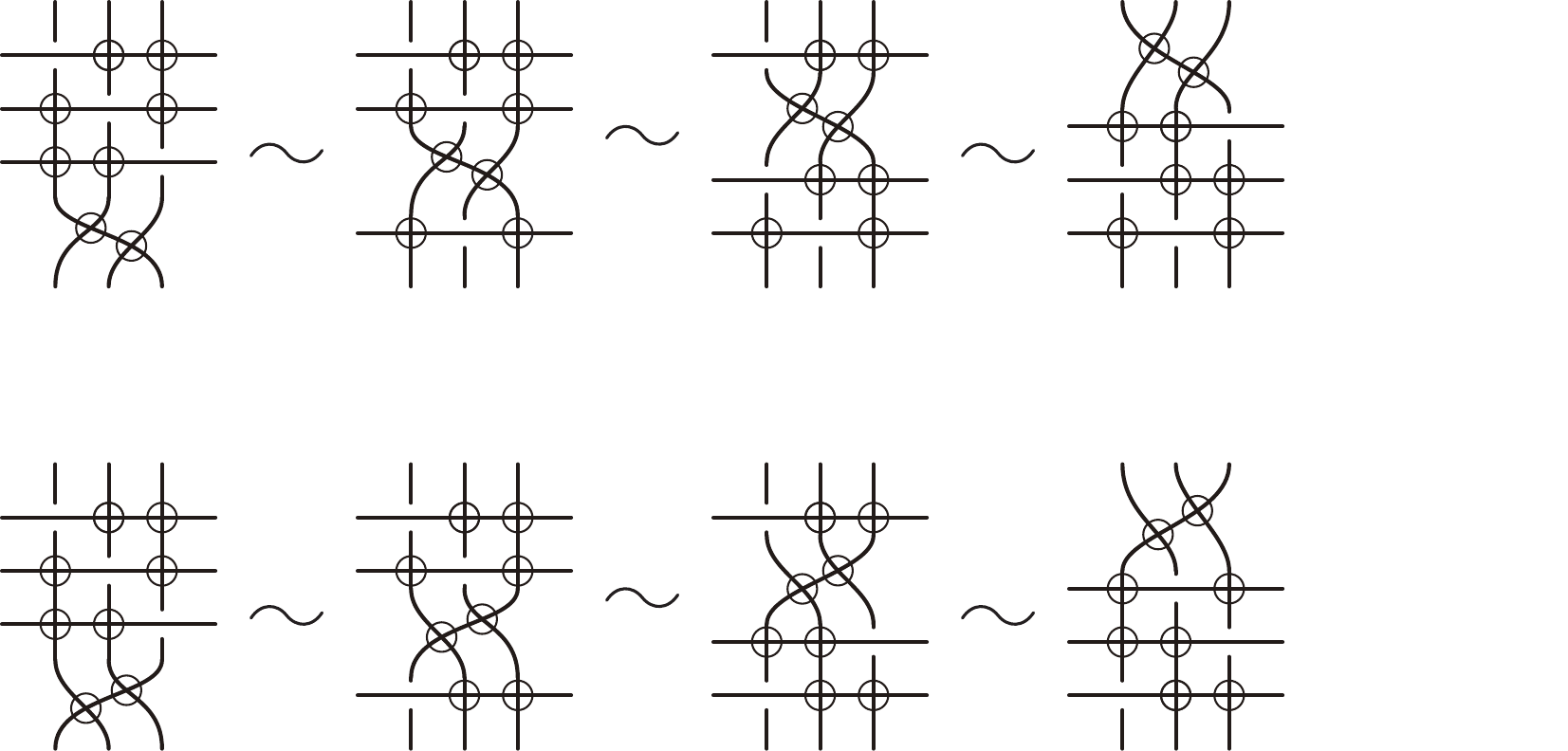}
    \put(0,159){\underline{$M_{1}(3)$ move}}
    \put(0,66.7){\underline{$M_{2}(3)$ move}}
  \end{overpic}
  
\vspace{3.5em}
  
  \begin{overpic}[width=11cm]{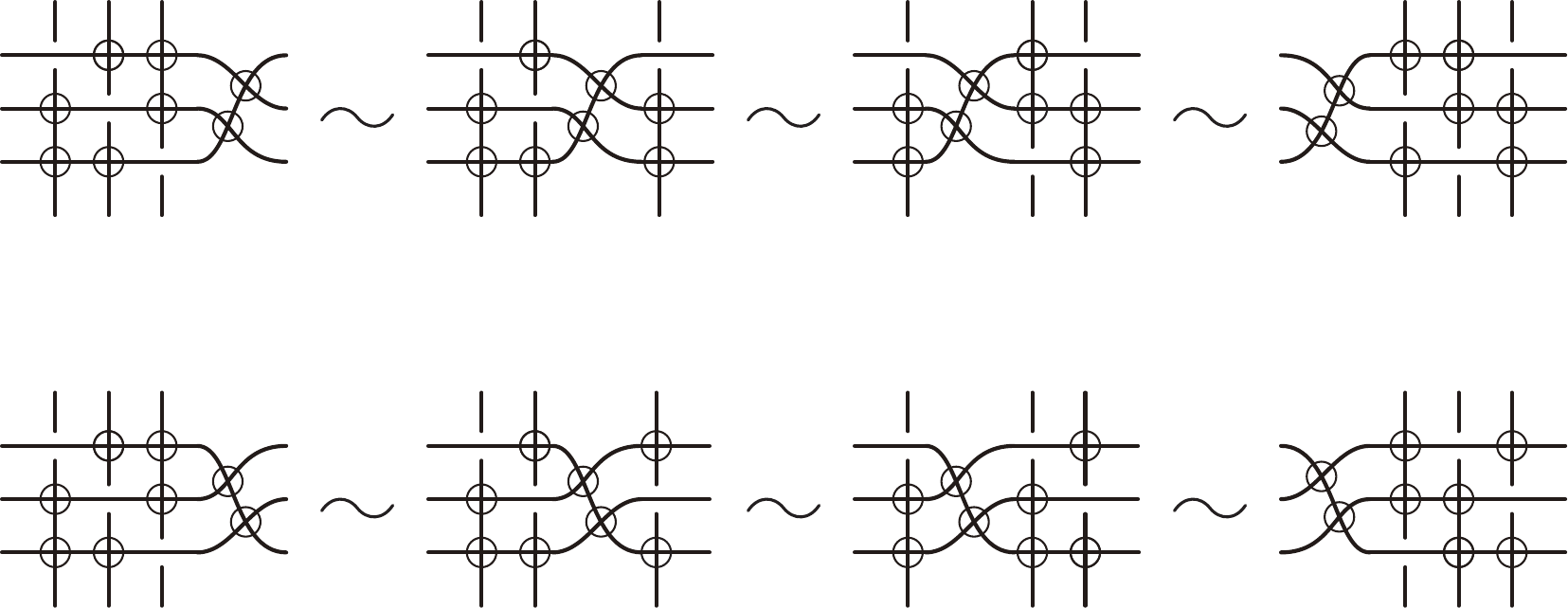}
    \put(0,130.5){\underline{$M_{3}(3)$ move}}
    \put(0,52.4){\underline{$M_{4}(3)$ move}}
  \end{overpic}
\caption{Proof of Lemma~\ref{lem-M1toM4} for $r=3$}
\label{pf-lem-M1toM4}
\end{figure}

Let $M_{5}(r)$, $M_{6}(r)$, $M_{7}(r)$ and $M_{8}(r)$ be the local moves 
obtained from $M_{1}(r)$, $M_{2}(r)$, $M_{3}(r)$ and $M_{4}(r)$, respectively, by changing the over/under information at every real crossing. 
For these moves, we can obtain a similar result to Lemma~\ref{lem-M1toM4} as follows:  

\begin{lemma}\label{lem-M5toM8}
Let $r$ be an integer with $r\geq2$. 
Each of the four moves $M_{5}(r)$--$M_{8}(r)$ is realized by a finite sequence of virtual Reidemeister moves. 
\qed
\end{lemma}

\begin{lemma}\label{lem-VR4}
If two virtual knot diagrams $D$ and $D'$ are related by a VR4 move, 
then $L(D;r)$ and $L(D';r)$ are equivalent for any integer $r$ with $r\geq2$. 
\end{lemma}

\begin{proof}
We only prove the result for the case $r=3$; 
the other cases are shown similarly. 

There are eight versions of oriented VR4 moves. 
Assume that $D$ and $D'$ are related by an oriented VR4 move as shown in the top of Figure~\ref{pf-lem-VR4}. 
The bottom of this figure shows that $L(D;3)$ and $L(D';3)$ are related by a finite sequence of an $M_{1}(3)$ move, an $M_{3}(3)$ move and several VR2, VR3 and VR4 moves. 
Therefore $L(D;3)$ and $L(D';3)$ are equivalent by Lemma~\ref{lem-M1toM4}. 
Using this lemma together with Lemma~\ref{lem-M5toM8}, we can similarly prove that if $D$ and $D'$ are related by one of the remaining seven  oriented VR4 moves, 
then $L(D;3)$ and $L(D';3)$ are equivalent. 
\end{proof}

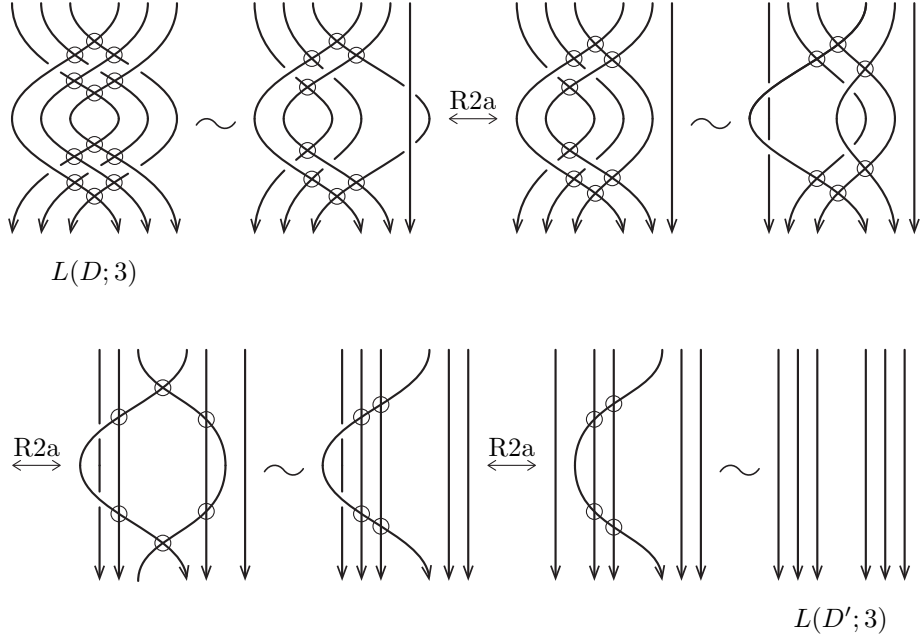
\begin{figure}[htbp]
\centering
  \begin{overpic}[width=12.5cm]{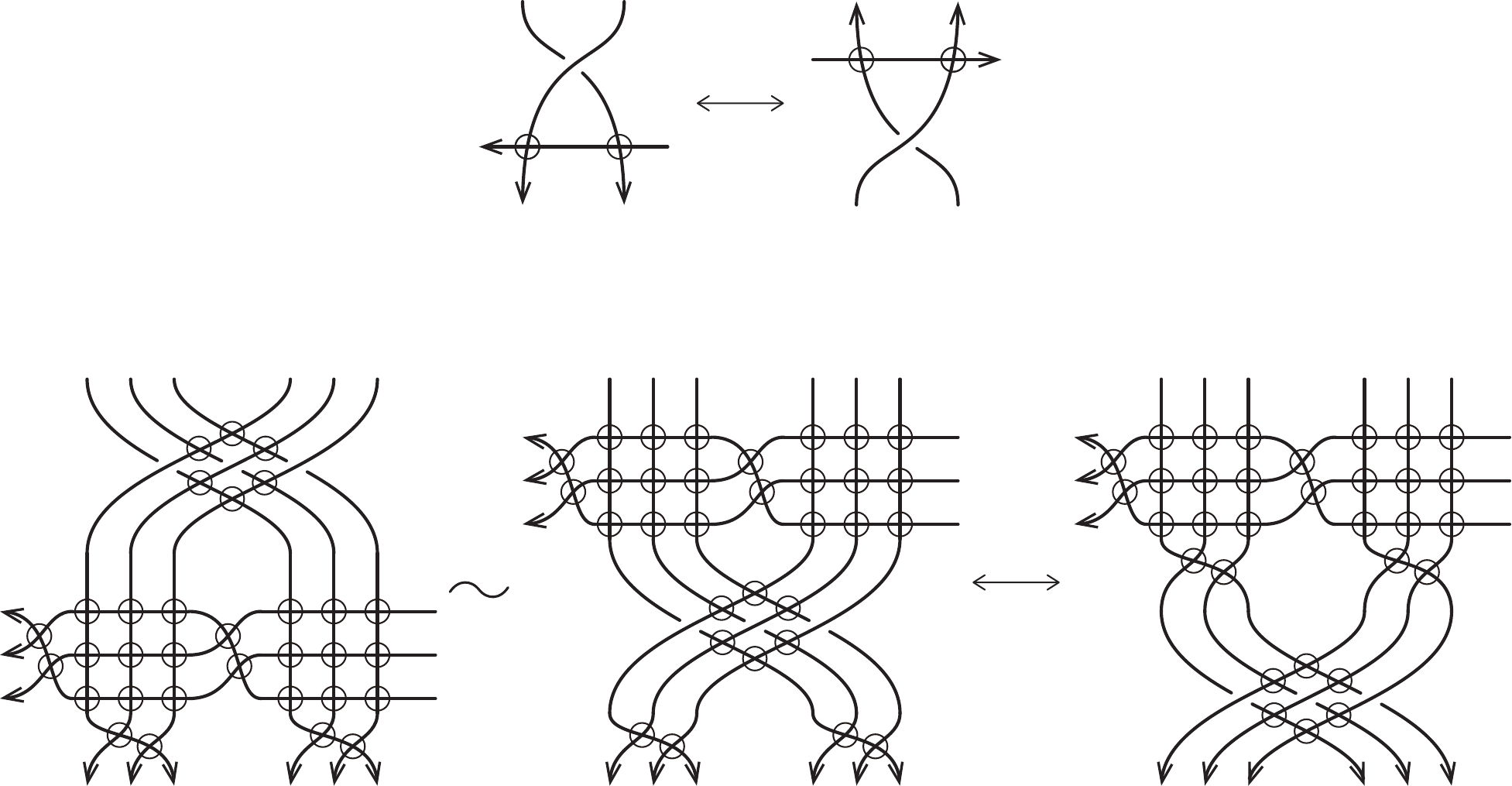}
    \put(131,123){$D$}
    \put(208,123){$D'$}
    \put(164,164.5){VR4}
    \put(226,53){$M_{1}(3)$}
    \put(226,36.5){$M_{3}(3)$}
    \put(39,-14){$L(D;3)$}
    \put(290,-14){$L(D';3)$}
  \end{overpic}
\vspace{1em}
\caption{Proof of Lemma~\ref{lem-VR4} for $r=3$}
\label{pf-lem-VR4}
\end{figure}

We are now ready to prove Theorem~\ref{thm-multiplexing}. 

\begin{proof}[Proof of Theorem~\ref{thm-multiplexing}]
It suffices to show that if $D$ and $D'$ are related by a generalized Reidemeister move, 
then $L(D;r)$ and $L(D';r)$ are equivalent for any $r\geq2$. 

Polyak's result in~\cite[Theorem~1.1]{Pol} says that
all the classical Reidemeister moves R1--R3 can be generated by the four moves R1a, R1b, R2a and R3a. 
Therefore if $D$ and $D'$ are related by a classical Reidemeister move, 
then $L(D;r)$ and $L(D';r)$ are equivalent by Lemmas~\ref{lem-R1}, \ref{lem-R2} and \ref{lem-R3}. 

We can see that 
if $D$ and $D'$ are related by one of the virtual Reidemeister moves VR1--VR3, 
then $L(D;r)$ and $L(D';r)$ are related by a finite sequence of moves VR1--VR3. 
Thus they are equivalent. 

If $D$ and $D'$ are related by a virtual Reidemeister move V4, 
then it follows from Lemma~\ref{lem-VR4} that $L(D;r)$ and $L(D';r)$ are equivalent. 
\end{proof}

\begin{definition}
Let $K$ be a virtual knot and $r$ an integer with $r\geq2$. 
The $r$-component virtual link represented by $L(D;r)$ constructed from a diagram $D$ of $K$ is called the \emph{$r$-multiplexed virtual link} of $K$.  
It is denoted by $L(K;r)$. 
\end{definition}

The well-definedness of $L(K;r)$ follows from Theorem~\ref{thm-multiplexing}.

\section{Invariants of multiplexed virtual links}\label{sec-inv}
This section is devoted to the proofs of Theorems~\ref{thm-lk}, \ref{thm-self-writhe} and \ref{thm-knot-type}. 

First we review the definition of the $n$-writhe of a virtual knot $K$. 
For a real crossing $c$ of a diagram $D$ of $K$, 
let $\rho$ be the oriented path that proceeds along $D$ from the overcrossing to the undercrossing at $c$. 
We call it the \emph{specified path} of $c$. 
When walking along $\rho$, we count $+1$ for a real crossing on $\rho$ if a string goes across $\rho$ at the real crossing from left to right; 
otherwise, we count $-1$ for it. 
See Figure~\ref{intersection-sign}. 
The \emph{index} of $c$ is the sum of such $\pm1$'s for all the real crossings on $\rho$, 
and it is denoted by $\mathrm{ind}(c)$. 
The indices of real crossings play an important role in this section; 
in fact, they will be used in the proofs of Theorems~\ref{thm-lk}, \ref{thm-self-writhe} and \ref{thm-knot-type}. 

\begin{figure}[htbp]
\centering
  \begin{overpic}[width=10cm]{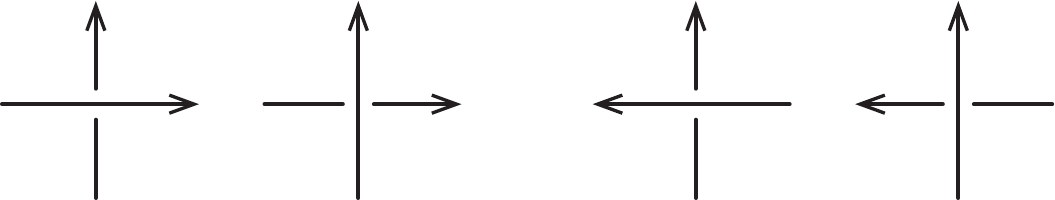}
    \put(29.1,31){$+1$}
    \put(23,-12){$\rho$}
    \put(100,31){$+1$}
    \put(93.6,-12){$\rho$}
    \put(191.5,31){$-1$}
    \put(185,-12){$\rho$}
    \put(262.5,31){$-1$}
    \put(255.5,-12){$\rho$}
  \end{overpic}
\vspace{1em}
\caption{Counting $\pm1$ for a real crossing on the specified path $\rho$}
\label{intersection-sign}
\end{figure}
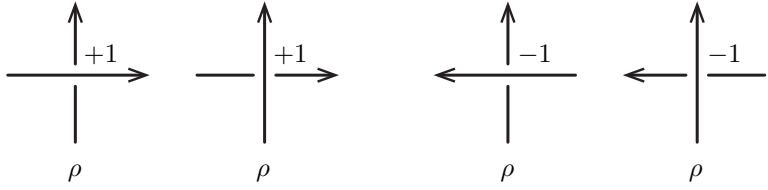

For an integer $n$, let $J_{n}(D)$ be the sum of the signs of all the real crossings of $D$ with index~$n$. 
Then $J_{n}(D)$ is an invariant of the virtual knot $K$ for any $n\neq0$~\cite[Lemma~2.3]{ST}. 
It is called the \emph{$n$-writhe} of $K$ and denoted by $J_{n}(K)$.

Next we give the definition of the $(i,j)$-linking number of an $r$-component virtual link $L=K_{1}\cup\dots\cup K_{r}$, 
which is one of the most fundamental invariants of $L$.  
Let $D=D_{1}\cup\dots\cup D_{r}$ be a diagram of $L$. 
For integers $i$ and $j$ with $1\leq i\neq j\leq r$, an \emph{$(i,j)$-crossing} of $D$ is a real crossing where $D_{i}$ passes over $D_{j}$. 
Denote by $\mathrm{Lk}(D_{i},D_{j})$ the sum of the signs of all the $(i,j)$-crossings of $D$. 
It is easy to see that $\mathrm{Lk}(D_{i},D_{j})$ is an invariant of $L$; cf.~\cite[Section~1.7]{GPV}. 
We call this invariant the \emph{$(i,j)$-linking number} of $L$ and denote it by $\mathrm{Lk}(K_{i},K_{j})$.

Then we give the proof of Theorem~\ref{thm-lk}. 

\begin{proof}[Proof of Theorem~\ref{thm-lk}]
Let $D$ be a diagram of $K$. 
For integers $i$ and $j$ with $1\leq i\neq j\leq r$, 
we consider an $(i,j)$-crossing $c$ of $L(D;r)=D_{1}\cup\dots\cup D_{r}$. 
Then there is a unique real crossing $c_{0}$ of $D$ producing~$c$. 
Moreover, the signs of $c$ and $c_{0}$ are the same. 
Hence the $(i,j)$-linking number $\mathrm{Lk}(D_{i},D_{j})$ is equal to the sum of the signs of all the real crossings of $D$ 
each of which yields an $(i,j)$-crossing of $L(D;r)$.

Now we will show that $\mathrm{ind}(c_{0})\equiv i-j\pmod{r}$. 
Let $\rho_{0}\subset D$ be the specified path of $c_{0}$. 
When walking along $\rho_{0}$, we count $+1$ for a virtual crossing on $\rho_{0}$ 
if a string goes across $\rho$ at the virtual crossing from left to right; 
otherwise, we count $-1$ for it.  
Denote by $\mathrm{ind}_{v}(c_{0})$ the sum of such $\pm1$'s for all the virtual crossings on $\rho_{0}$. 
Since the $i$th component $D_{i}$ of $L(D;r)$ passes over the $j$th component $D_{j}$ at $c$, 
we have 
\[
\mathrm{ind}_{v}(c_{0})\equiv j-i\pmod{r}.
\] 
In fact, when a real crossing $x_{0}$ of $D$ has $\mathrm{ind}_{v}(x_{0})=k$, 
at any real crossing of $L(D;r)$ derived from $x_{0}$, 
$D_{l}$ passes over $D_{l+k}$ for some $l\in\{1,\dots,r\}$, 
where $l+k$ is taken modulo $r$. 
See Figure~\ref{gap} for an example with $\mathrm{ind}_{v}(x_{0})=2$ and $r=3$. 

\begin{figure}[htbp]
\vspace{1em}
\centering
  \begin{overpic}[width=12cm]{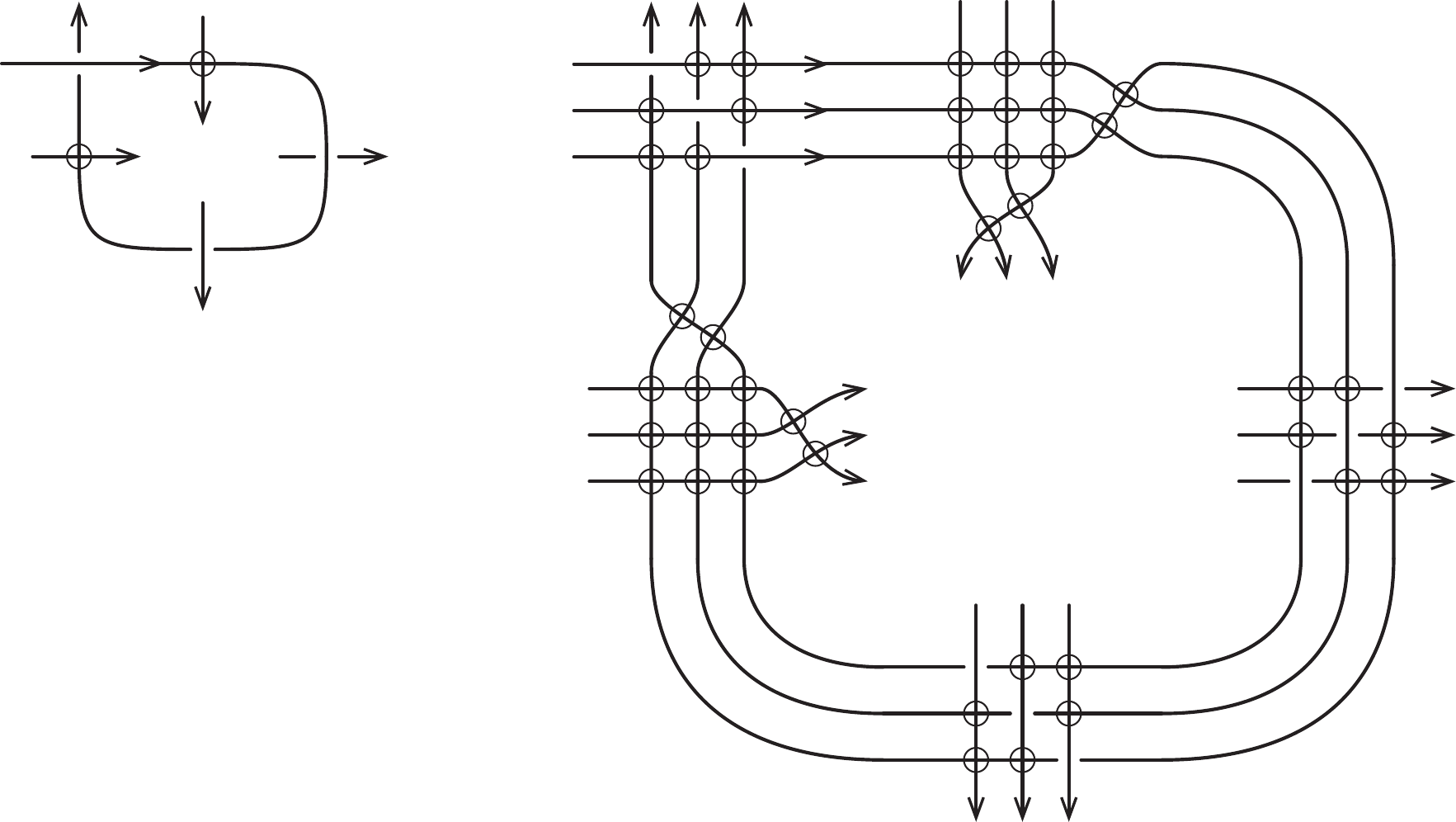}
    \put(36,106){$D$}
    \put(6,170){$x_{0}$}
    \put(144.5,196){$D_{2}$}
    \put(158.5,196){$D_{1}$}
    \put(172,196){$D_{3}$}
    \put(119,175){$D_{3}$}
    \put(119,164){$D_{2}$}
    \put(119,153){$D_{1}$}
    \put(187,-14){$L(D;3)=D_{1}\cup D_{2}\cup D_{3}$}
  \end{overpic}
  \vspace{1em}
  \caption{Proof of Theorem~\ref{thm-lk}}
\label{gap}
\end{figure}
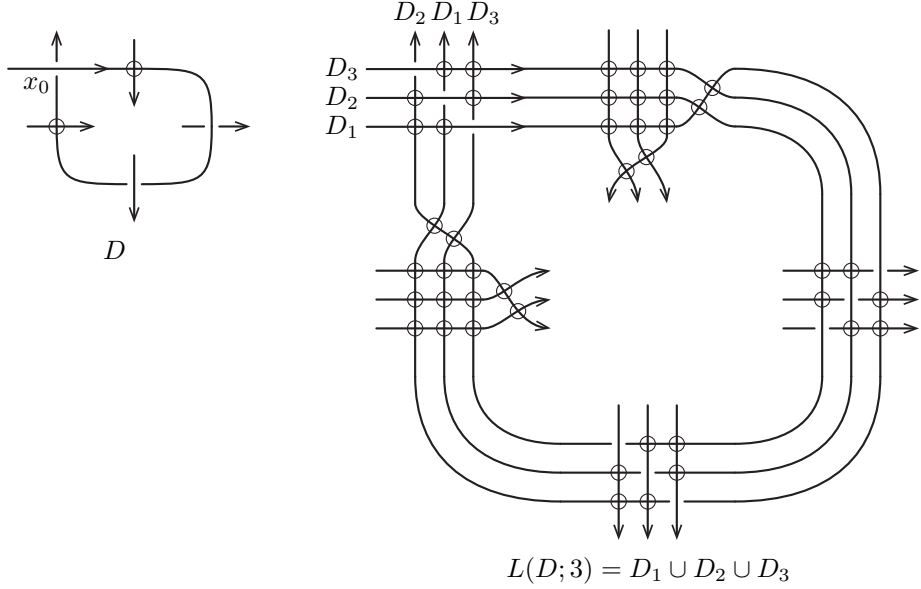

With respect to the orientation of $\rho_{0}$, 
the number of strings going across $\rho_{0}$ at real/virtual crossings from left to right is equal to that of strings going across $\rho_{0}$ from right to left. 
Since we have $\mathrm{ind}(c_{0})+\mathrm{ind}_{v}(c_{0})=0$, 
it follows that 
\[
\mathrm{ind}(c_{0})\equiv i-j\pmod{r}. 
\] 
Thus $\mathrm{Lk}(D_{i},D_{j})$ is equal to the sum of the signs of all the real crossings of $D$ 
whose indices are congruent to $i-j$ modulo~$r$. 
\end{proof}

Again we consider an $r$-component virtual link $L=K_{1}\cup\dots\cup K_{r}$ and its diagram $D=D_{1}\cup\dots\cup D_{r}$ 
in order to define the $i$th $n$-writhe of $L$, 
which is an extension of the $n$-writhe of virtual knots.  
For an integer $i$ with $1\leq i\leq r$, 
a \emph{self-crossing} of $D_{i}$ is a real/virtual crossing of $D$ involving two strings of $D_{i}$. 
Let $c$ be a real self-crossing of $D_{i}$ and 
$\rho\subset D_{i}$ the specified path of $c$, 
which proceeds along $D_{i}$ from the overcrossing to the undercrossing at $c$. 
When walking along $\rho$, we count $\pm1$ for every real crossing on $\rho$ as explained above. 
See Figure~\ref{intersection-sign} again. 
The \emph{index of $c$ in $D$} is the sum of such $\pm1$'s for all the real crossings on $\rho$ and denoted by $\mathrm{ind}(D;c)$. 

For an integer $n$, 
we denote by $J_{n}^{i}(D)$ the sum of the signs of all the real self-crossings of $D_{i}$ with index $n$ in $D$. 
It was shown by Xu~\cite{Xu} that $J_{n}^{i}(D)$ is an invariant of the virtual link $L$ 
for any nonzero integer $n$ with $n\neq\lambda_{i}(L)$. 
Refer also to~\cite[Section~5]{NNS-writhe}. 
Here $\lambda_{i}(L)$ is an invariant of $L$ defined by 
\[
\lambda_{i}(L)=\displaystyle\sum_{1\leq j\neq i\leq r}\left(\mathrm{Lk}(K_{j},K_{i})-\mathrm{Lk}(K_{i},K_{j})\right). 
\] 
We call $J_{n}^{i}(D)$ the \emph{$i$th $n$-writhe of $L$} and denote it by $J_{n}^{i}(L)$.
We emphasize that $J_{n}^{i}(L)$ differs from the original $n$-writhe $J_{n}(K_{i})$ of $K_{i}$ in the sense that $J_{n}^{i}(L)$ is an invariant of $L$ rather than just $K_{i}$.

Now we turn to the proof of Theorem~\ref{thm-self-writhe}. 

\begin{proof}[Proof of Theorem~\ref{thm-self-writhe}]
For a diagram $D$ of $K$ and an integer $i$ with $1\leq i\leq r$, 
let $c$ be a real self-crossing of the $i$th component $D_{i}$ of $L(D;r)$. 
Set $n=\mathrm{ind}(L(D;r);c)$. 
Then there is a unique real crossing $c_{0}$ of $D$ producing $c$. 
Moreover, the signs of $c$ and $c_{0}$ are the same. 
Hence the $i$th $n$-writhe $J_{n}^{i}(L(D;r))$ is equal to the sum of the signs of all the real crossings of $D$ 
each of which yields a real self-crossing of $D_{i}$ with index $n$ in $L(D;r)$. 

Now we will show that $n=\mathrm{ind}(c_{0})\equiv0\pmod{r}$. 
Since two strings of $D_{i}$ cross at $c$, 
we have 
\[
\mathrm{ind}_{v}(c_{0})\equiv0\pmod{r}
\] 
by a similar argument to that in the proof of Theorem~\ref{thm-lk}. 
Moreover, it follows from the equation $\mathrm{ind}(c_{0})+\mathrm{ind}_{v}(c_{0})=0$ that 
\[
\mathrm{ind}(c_{0})\equiv0\pmod{r}. 
\] 

The number of real crossings of $L(D;r)$ on the specified path $\rho$ of $c$ 
is equal to that of real crossings of $D$ on the specified path $\rho_{0}$ of $c_{0}$. 
If we count $\e\in\{\pm1\}$ for a real crossing of $L(D;r)$ on $\rho$ when walking along $\rho$, then we count the same sign $\e$ for the corresponding real crossing of $D$ on $\rho_{0}$ when walking along $\rho_{0}$. 
Therefore it follows that 
\[
n=\mathrm{ind}(L(D;r);c)=\mathrm{ind}(c_{0})\equiv0\pmod{r}.
\] 
This induces that $J_{n}^{i}(L(D;r))$ is equal to the sum of the signs of all the real crossings of $D$ 
whose indices are congruent to zero modulo~$r$. 
\end{proof}

For a positive integer $r$, we define the $r$th covering of a virtual knot $K$.  
Let $D$ be a diagram of $K$. 
We denote by $D^{(r)}$ the virtual knot diagram obtained by replacing all the real crossings of $D$ 
whose indices are not divisible by~$r$ 
with virtual crossings. 
It is known that 
if two virtual knot diagrams $D$ and $D'$ are equivalent, then so are $D^{(r)}$ and $(D')^{(r)}$ for any $r\geq1$; cf.~\cite{IK,NNS,Tur}. 
That is, the equivalence class of $D^{(r)}$ is an invariant of $K$. 
The virtual knot represented by $D^{(r)}$ is called the \emph{$r$th covering} of $K$ and denoted by $K^{(r)}$.

We conclude this section with the proof of Theorem~\ref{thm-knot-type}. 

\begin{proof}[Proof of Theorem~\ref{thm-knot-type}]
Let $D$ be a diagram of $K$. 
By the definition of the $r$-multiplexing of $D$, 
for any $i\in\{1,\ldots,r\}$, 
the number of self-crossings of the $i$th component $D_{i}$ of $L(D;r)$ 
is equal to that of crossings of $D$.  
Hence there is a bijection from the set of self-crossings of $D_{i}$ 
to that of crossings of $D$. 
For a self-crossing $c$ of $D_{i}$, 
we will determine the corresponding crossing of $D$. 

If $c$ is a real crossing, 
then there is a unique real crossing $c_{0}$ of $D$ producing~$c$. 
Moreover, we have $\mathrm{ind}(c_{0})\equiv0\pmod{r}$ 
as seen in the proof of Theorem~\ref{thm-self-writhe}. 
Therefore the set of real self-crossings of $D_{i}$ corresponds bijectively to 
that of real crossings of $D$ whose indices are congruent to zero modulo~$r$. 
If $c$ is a virtual crossing, 
then there is a unique crossing $c'_{0}$ of $D$ producing $c$. 
It is either a real crossing with $\mathrm{ind}(c'_{0})\not\equiv0\pmod{r}$ 
or a virtual crossing. 
Therefore the set of virtual self-crossings of $D_{i}$ corresponds bijectively to 
the union of the set of real crossings of $D$ whose indices are not congruent to zero modulo~$r$ 
and that of virtual crossings of $D$. 
Thus the virtual knot diagram $D_{i}$ is obtained from $D$ by replacing all the real crossings of $D$ whose indices are not divisible by $r$ with virtual crossings; 
that is, $D_{i}=D^{(r)}$. 
\end{proof}

\section{Virtual colorings}\label{sec-coloring}
In this section we define virtual $n$-colorings of a virtual link diagram 
as a special case of Manturov's virtual quandle colorings (Definition~\ref{def-vcoloring}), 
and prove Theorem~\ref{thm-coloring}. 

We begin by recalling the definition of (classical) $n$-colorings of a virtual link diagram $D$. 
An \emph{arc} of $D$ proceeds from an undercrossing to the next one, where overcrossings and virtual crossings are ignored. 
We say that a map $C$ from the set of arcs of $D$ to a cyclic group $\mathbb{Z}/n\mathbb{Z}$ is 
an \emph{$n$-coloring} of $D$ if at every real crossing 
it satisfies 
\[
x+z=2y,
\] 
called \emph{Fox's $n$-coloring condition}, 
where 
$x$, $y$ and $z$ are the elements in $\mathbb{Z}/n\mathbb{Z}$ assigned by $C$ to three arcs at the real crossing as shown in Figure~\ref{classical-coloring}. 

\begin{figure}[htbp]
\centering
\vspace{1em}
  \begin{overpic}[width=1.5cm]{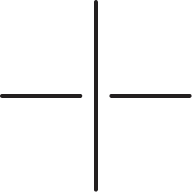}
    \put(-9,19.5){$x$}
    \put(46.2,19.5){$z$}
    \put(19,48.5){$y$}
  \end{overpic}
  \caption{Fox's $n$-coloring condition $x+z=2y$ at a real crossing}
\label{classical-coloring}
\end{figure}

Denote by $\mathrm{Col}_{n}(D)$ the set of $n$-colorings of $D$. 
It is easy to see that 
if two virtual link diagrams $D$ and $D'$ are equivalent, 
then there is a bijection between $\mathrm{Col}_{n}(D)$ and $\mathrm{Col}_{n}(D')$; 
cf. \cite[Section~4]{Kau}. 
In particular, the cardinality of $\mathrm{Col}_{n}(D)$ is an invariant of the virtual link represented by $D$. 

Let $D$ be a virtual knot diagram. 
Now we consider a certain $n$-coloring $C$ of the $2$-multiplexed virtual link diagram $L(D;2)$ of $D$ as follows. 
Let $(\alpha,\alpha')$ be any pair of parallel arcs of $L(D;2)$ except near the sets of crossings of $L(D;2)$ derived from the real crossings of $D$. 
Then $C$ satisfies $C(\alpha')=-C(\alpha)$. 
Denote by $\mathrm{Col}_{n}^{0}(L(D;2))\subset \mathrm{Col}_{n}(L(D;2))$ the set of such $n$-colorings  of $L(D;2)$. 
The right of Figure~\ref{ex-coloring} shows an example of a $3$-coloring of $L(D;2)$ in $\mathrm{Col}_{3}^{0}(L(D;2))$, 
where this $2$-multiplexed virtual link diagram $L(D;2)$ is constructed from the virtual knot diagram $D$ in the left of this figure. 
(The elements $0,1,2\in\mathbb{Z}/3\mathbb{Z}$ assigned to $D$ will be explained later.)

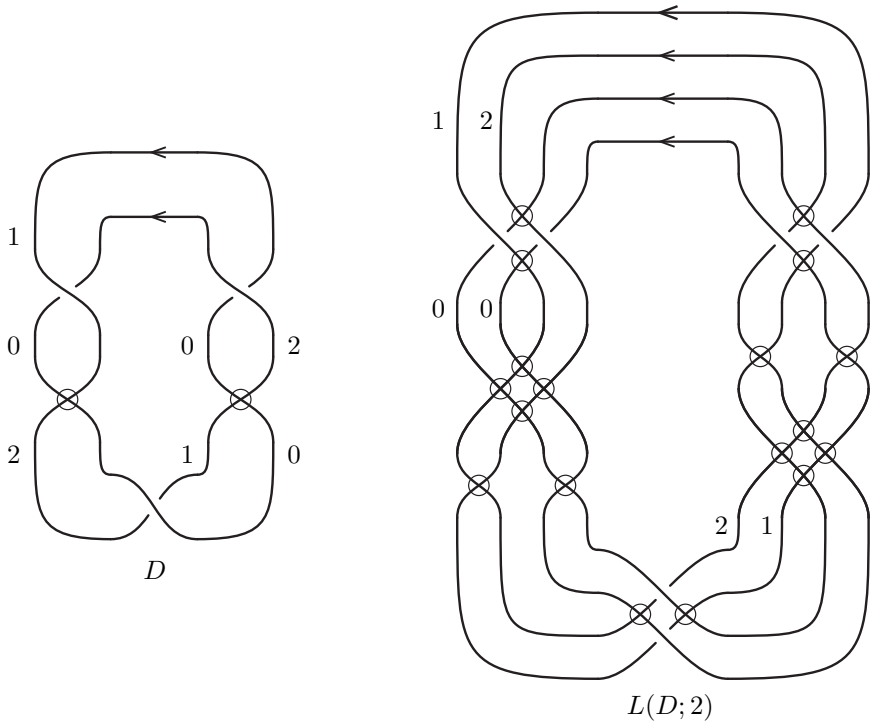
\begin{figure}[htbp]
\centering
  \begin{overpic}[width=11cm]{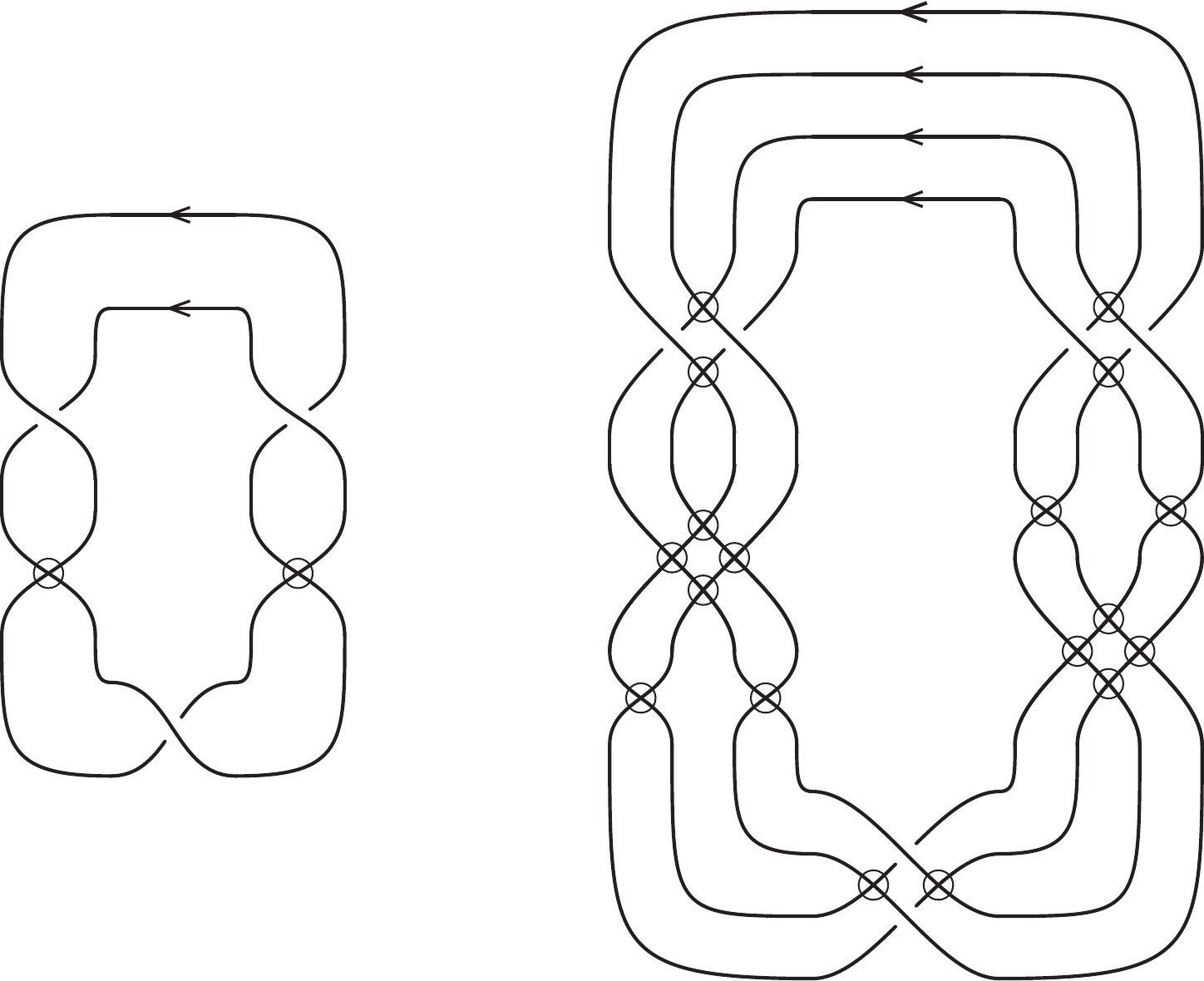}
    \put(41,38){$D$}
    \put(222,-13){$L(D;2)$}
    \put(-10,164){1}
    \put(-10,123){0}
    \put(-10,82){2}
    \put(55,123){0}
    \put(55,82){1}
    \put(95,123){2}
    \put(95,82){0}
    \put(149,208){1}
    \put(167,208){2}
    \put(149,137){0}
    \put(167,137){0}
    \put(255,55){2}
    \put(272,55){1}
  \end{overpic}
\vspace{1em}
\caption{A virtual $3$-coloring of a virtual knot diagram $D$ and a $3$-coloring of the $2$-multiplexed virtual link diagram $L(D;2)$}
\label{ex-coloring}
\end{figure}

A \emph{quandle} \cite{Joy}, also known as a \emph{distributive groupoid} \cite{Mat}, is a set $Q$ with a binary operation $\triangleright$ satisfying three axioms, which correspond to the classical Reidemeister moves R1--R3. 
Several invariants of classical links are derived from \emph{quandle colorings} of diagrams. 
A quandle coloring of a classical link diagram straightforwardly extends to that of a virtual link diagram by ignoring the virtual crossings. 
A typical example of a quandle coloring is an $n$-coloring defined above.

A \emph{virtual quandle}, introduced by Manturov~\cite{Man}, 
is a quandle $(Q,\triangleright)$ equipped with an invertible unary operation $f$ such that $\triangleright$ is distributive with respect to $f$.
Adding two new relations to every virtual crossing by $f$, he defined \emph{virtual quandle colorings} of a virtual link diagram. 
As well as quandle colorings, several invariants of virtual links are derived from virtual quandle colorings. 
Refer to \cite{Man} for more details. 

As a special case of virtual quandle colorings, we define virtual $n$-colorings of a virtual link diagram $D$ as follows. 
Divide the arcs of $D$ at the virtual crossings of $D$. 
The resulting subarcs are called \emph{virtual arcs} of $D$. 
Note that a virtual arc may contain real overcrossings.

\begin{definition}\label{def-vcoloring}
A map $C$ from the set of virtual arcs of $D$ to $\mathbb{Z}/n\mathbb{Z}$ 
is called a \emph{virtual $n$-coloring} of $D$ 
if it satisfies Fox's $n$-coloring condition at every real crossing, 
and 
\[
z=-x\quad \mbox{and}\quad w=-y
\] 
at every virtual crossing, 
where $x$, $y$, $z$ and $w$ are the elements assigned by $C$ to four virtual arcs at the virtual crossing as shown in Figure~\ref{virtual-coloring}. 
\end{definition}

\begin{figure}[htbp]
\centering
\vspace{1em}
  \begin{overpic}[width=1.5cm]{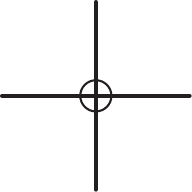}
    \put(-9,19.5){$x$}
    \put(46.2,19.5){$z$}
    \put(19,48.5){$y$}
    \put(17,-9.2){$w$}
  \end{overpic}
\vspace{1em}
\caption{Two conditions $z=-x$ and $w=-y$ at a virtual crossing}
\label{virtual-coloring}
\end{figure}
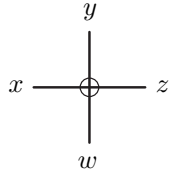

The left of Figure~\ref{ex-coloring} shows an example of a virtual $3$-coloring. 

Denote by $\mathrm{Col}_{n}^{v}(D)$ the set of virtual $n$-colorings of $D$. 
It follows from~\cite[Theorem 3]{Man} that if two virtual link diagrams $D$ and $D'$ are equivalent, then there is a bijection between $\mathrm{Col}_{n}^{v}(D)$ and $\mathrm{Col}_{n}^{v}(D')$. 
In particular, the cardinality of $\mathrm{Col}_{n}^{v}(D)$ is an invariant of the virtual link represented by $D$.

\begin{remark}
A \emph{forbidden underpass move}, which is one of the two forbidden moves in virtual knot theory~\cite[Section 1.3]{GPV}, is a local move passing a string under a virtual crossing as shown in Figure~\ref{FU}.  
We can see that if two virtual link diagrams $D$ and $D'$ are related by a forbidden underpass move, then there is a bijection between $\mathrm{Col}_{n}^{v}(D)$ and $\mathrm{Col}_{n}^{v}(D')$. 
\end{remark}

\begin{figure}[htbp]
\centering
  \begin{overpic}[width=4cm]{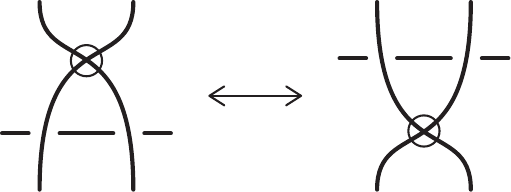}
      \end{overpic}
  \caption{A forbidden underpass move}
\label{FU}
\end{figure}

We conclude this paper with the proof of Theorem~\ref{thm-coloring}. 

\begin{proof}[Proof of Theorem~\ref{thm-coloring}]
We may assume that the number of crossings of $D$ is one or more, 
applying an R1 move if necessary. 
Let $c_{1},\dots,c_{k}$ $(k\geq1)$ be the crossings of $D$. 
Identifying $c_{1},\dots,c_{k}$ with vertices, we regard $D$ as a $4$-valent graph. 
Then we have $2k$ edges of $D$ and denote them by $e_{1},\dots,e_{2k}$. 
Let $\Delta_{1},\dots,\Delta_{k}$ be $2$-disks on the same plane as $L(D;2)$ such that for each $i\in\{1,\ldots,k\}$, 
$\Delta_{i}$ contains the set of crossings of $L(D;2)$ derived from $c_{i}$ as shown in Figure~\ref{pf-thm-coloring}. 

\begin{figure}[htbp]
\centering
\vspace{2em}
  \begin{overpic}[width=8cm]{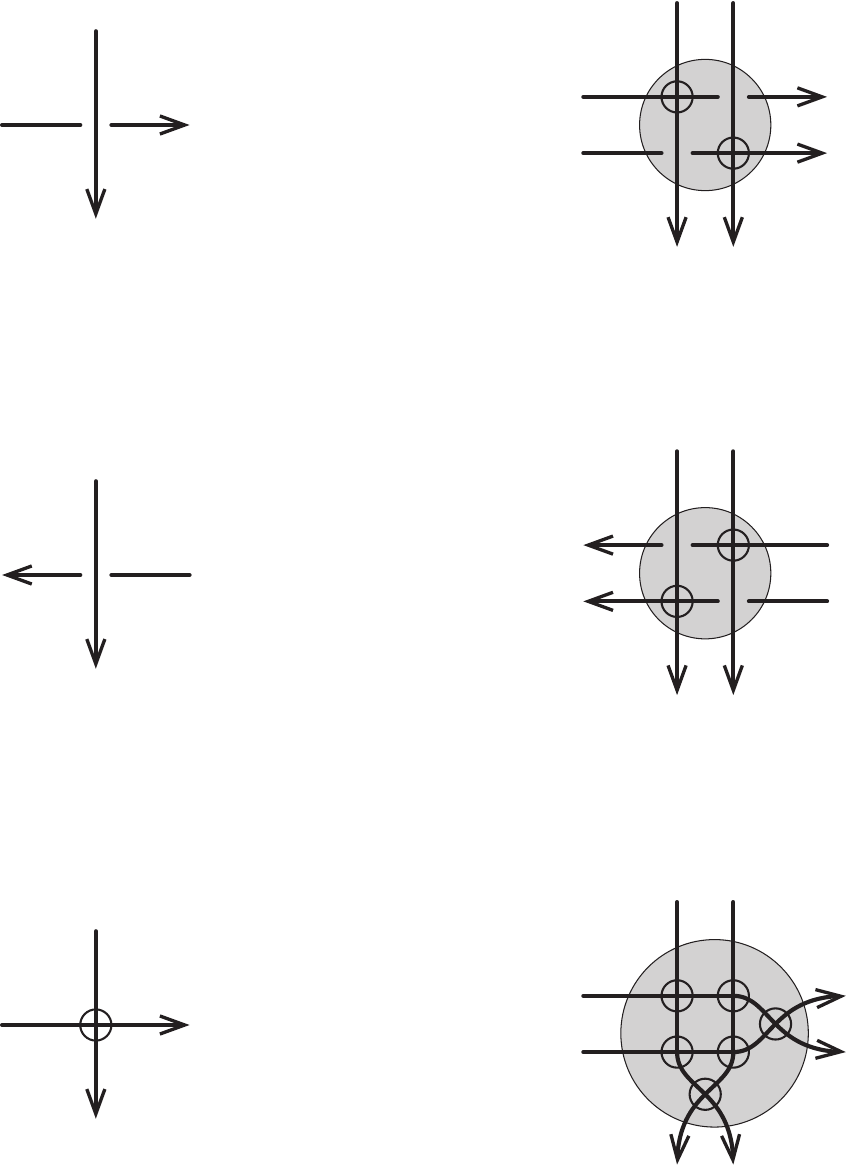}
    \put(-20,326.7){\underline{$c_{i}$ is a positive crossing.}}
    \put(-20,206){\underline{$c_{i}$ is a negative crossing.}}
    \put(-20,84.8){\underline{$c_{i}$ is a virtual crossing.}}
    \put(21,230){$D$}
    \put(13.5,269.8){$c_{i}$}
    \put(-9.25,277.3){$x$}
    \put(54.5,277.3){$2y-x$}
    \put(22.8,311.5){$y$}
    \put(22.8,246){$y$}
    \put(21,109){$D$}
    \put(13.5,148.7){$c_{i}$}
    \put(-9.25,156.2){$x$}
    \put(54.5,156.2){$2y-x$}
    \put(22.8,190.5){$y$}
    \put(22.8,125){$y$}
    \put(21,-12){$D$}
    \put(13.5,27.6){$c_{i}$}
    \put(-9.25,35.1){$x$}
    \put(54.5,35.1){$-x$}
    \put(22.8,69.5){$y$}
    \put(18.2,4){$-y$}
    \put(173.5,221.5){$L(D;2)$}
    \put(202.4,259.5){$\Delta_{i}$}
    \put(139.3,284.85){$-x$}
    \put(147,269.7){$x$}
    \put(226.5,284.85){$-(2y-x)$}
    \put(226.5,269.7){$2y-x$}
    \put(179.2,319.2){$y$}
    \put(189.85,319.2){$-y$}
    \put(179.2,238.5){$y$}
    \put(189.85,238.5){$-y$}
    \put(173.5,101){$L(D;2)$}
    \put(202.4,139){$\Delta_{i}$}
    \put(139.3,149.1){$-x$}
    \put(147,164.2){$x$}
    \put(226.5,149.1){$-(2y-x)$}
    \put(226.5,164.2){$2y-x$}
    \put(179.2,198.5){$y$}
    \put(189.85,198.5){$-y$}
    \put(179.2,117.5){$y$}
    \put(189.85,117.5){$-y$}
    \put(173.5,-25){$L(D;2)$}
    \put(212,12){$\Delta_{i}$}
    \put(139.3,42.9){$-x$}
    \put(147,27.8){$x$}
    \put(231.5,27.8){$-x$}
    \put(231.5,42.9){$x$}
    \put(179.2,77.3){$y$}
    \put(189.85,77.3){$-y$}
    \put(174.5,-8.5){$-y$}
    \put(194.5,-8.5){$y$}
  \end{overpic}
    \vspace{2em}
  \caption{A bijection between $\mathrm{Col}_{n}^{v}(D)$ and $\mathrm{Col}_{n}^{0}(L(D;2))$}
\label{pf-thm-coloring}
\end{figure}

Removing the intersection $L(D;2)\cap(\Delta_{1}\cup\dots\cup\Delta_{k})$ from $L(D;2)$, we obtain $2k$ pairs of parallel arcs of $L(D;2)$ and denote them by $(\alpha_{1},\alpha'_{1}),\dots,(\alpha_{2k},\alpha'_{2k})$. 
For each $j\in\{1,\ldots,2k\}$, we assume that a pair $(\alpha_{j},\alpha'_{j})$ originates from $e_{j}$, 
and moreover the two arcs $\alpha_{j}$ and $\alpha'_{j}$ are located on the right and left sides with respect to the orientation of $L(D;2)$, respectively. 

Now we define a map $\Psi$ from $\mathrm{Col}_{n}^{v}(D)$ to $\mathrm{Col}_{n}^{0}(L(D;2))$ as follows. 
Take a virtual $n$-coloring $C\in\mathrm{Col}_{n}^{v}(D)$. 
When an edge $e_{j}$ of $D$ receives an element $x\in\mathbb{Z}/n\mathbb{Z}$ by $C$ as a part of a virtual arc, 
we assign $x$ to $\alpha_{j}$ and $-x$ to $\alpha'_{j}$. 
Then we obtain a unique $n$-coloring of $L(D;2)$ in $\mathrm{Col}_{n}^{0}(L(D;2))$. 
See Figure~\ref{pf-thm-coloring} again. 
Hence the map $\Psi$ is injective. 

Conversely, let $C'\in\mathrm{Col}_{n}^{0}(L(D;2))$ be any $n$-coloring of $L(D;2)$. 
For every pair $(\alpha_{j},\alpha'_{j})$ of parallel arcs of $L(D;2)$, we pay attention to $\alpha_{j}$, 
which is located on the right side with respect to the orientation of $L(D;2)$. 
When $\alpha_{j}$ receives an element $x\in\mathbb{Z}/n\mathbb{Z}$ by $C'$, 
we assign the same element $x$ to $e_{j}$. 
Then we obtain a virtual $n$-coloring of $D$. 
Since the map $\Psi$ sends this coloring to $C'\in\mathrm{Col}_{n}^{0}(L(D;2))$, it is surjective. 
Therefore we have a bijection $\Psi$ from $\mathrm{Col}_{n}^{v}(D)$ to $\mathrm{Col}_{n}^{0}(L(D;2))$. 
\end{proof}




\begin{thebibliography}{99}
\bibitem{GPV} 
M. Goussarov, M. Polyak and O. Viro, 
\textit{Finite-type invariants of classical and virtual knots}, 
Topology {\bf 39} (2000), no. 5, 1045--1068. 

\bibitem{IK} 
Y. H. Im and S. Kim, 
\textit{A sequence of polynomial invariants for Gauss diagrams}. 
J. Knot Theory Ramifications \textbf{26} (2017), no. 7, 1750039, 9 pp. 

\bibitem{Joy} 
D. Joyce, 
\textit{A classifying invariant of knots, the knot quandle}, 
J. Pure Appl. Algebra \textbf{23} (1982), no. 1, 37--65.

\bibitem{Kam19-talk} 
N. Kamada, 
\textit{A conversion of virtual links into mod $p$ almost classical virtual links}, 
talk at The 14th East Asian Conference on Geometric Topology at Peking University in January 2019. 

\bibitem{Kam19} 
N. Kamada, 
\textit{Cyclic coverings of virtual link diagrams}, 
Internat. J. Math. \textbf{30} (2019), no. 14, 1950072, 16 pp. 

\bibitem{Kau} 
L. H. Kauffman, 
\textit{Virtual knot theory}, 
European J. Combin. \textbf{20} (1999), no. 7, 663--690. 

\bibitem{Man} 
V. O. Manturov, 
\textit{On invariants of virtual links}, 
Acta Appl. Math. \textbf{72} (2002), no. 3, 295--309. 

\bibitem{Mat} 
S. V. Matveev, 
\textit{Distributive groupoids in knot theory} (Russian), 
Mat. Sb. (N.S.) \textbf{119}(\textbf{161}) (1982), no. 1, 78--88, 160. 

\bibitem{NNS-writhe}
T. Nakamura, Y. Nakanishi and S. Satoh, 
\textit{Writhe polynomials and shell moves for virtual knots and links}, 
European J. Combin. \textbf{84} (2020), 103033, 24 pp.

\bibitem{NNS}
T. Nakamura, Y. Nakanishi and S. Satoh, 
\textit{A note on coverings of virtual knots}, 
J. Knot Theory Ramifications \textbf{29} (2020), no. 8, 1971002, 11 pp. 


\bibitem{Pol}
M. Polyak, 
\textit{Minimal generating sets of Reidemeister moves}, 
Quantum Topol. \textbf{1} (2010), no.~4, 399--411.

\bibitem{ST} 
S. Satoh and K. Taniguchi, 
\textit{The writhes of a virtual knot}, 
Fund. Math. \textbf{225} (2014), no. 1, 327--342. 

\bibitem{Tur} 
V. Turaev, 
\textit{Virtual strings}, 
Ann. Inst. Fourier (Grenoble) \textbf{54} (2004), no. 7, 2455--2525 (2005).

\bibitem{Xu} 
M. Xu, 
\textit{Writhe polynomial for virtual links}, 
arXiv:1812.05234.
\end{thebibliography}
\end{document}